\documentclass[12pt,letterpaper,reqno]{amsart}
\numberwithin{equation}{section}

\usepackage[mathbf]{euler} 
\usepackage{amsrefs}
\usepackage{amsthm}
\usepackage{amssymb}
\usepackage{mathrsfs}
\usepackage{amsmath}
\usepackage{latexsym}
\usepackage[CJKbookmarks=true]{hyperref}
\usepackage{graphicx,tikz}
\usepackage[all]{xy}
\usepackage{enumerate}

\newtheorem{Th}{Theorem}[section]
\newtheorem{Lemma}[Th]{Lemma}
\newtheorem{Coro}[Th]{Corollary}
\newtheorem{Prop}[Th]{Proposition}
\newtheorem{Eg}[Th]{Example}
\newtheorem{Rmk}[Th]{Remark}

\def\G{\mathfrak{G}}
\def\S{\mathfrak{S}}

\def\PP{\operatorname{PP}}
\def\wt{\operatorname{wt}}
\let\oldvdots\vdots
\def\vdots{\raisebox{0pc}[0.65pc]{\(\oldvdots\)}}
\let\oldddots\ddots
\def\ddots{\raisebox{0pc}[0.65pc]{\(\oldddots\)}}

\newcommand{\BPD}[2][1pc]{%
\setlength{\unitlength}{#1}
\def\FF{%
    \qbezier(0.5,0)(0.5,0.2)(0.5,0.2)
    \qbezier(1,0.5)(0.8,0.5)(0.8,0.5)
    \qbezier(0.8,0.5)(0.5,0.5)(0.5,0.2)
}
\def\JJ{%
    \qbezier(0.5,1)(0.5,0.8)(0.5,0.8)
    \qbezier(0,0.5)(0.2,0.5)(0.2,0.5)
    \qbezier(0.5,0.8)(0.5,0.5)(0.2,0.5)
    }
\def\II{%
    \qbezier(0.5,0)(0.5,0.5)(0.5,1)}%
\def\HH{%
    \qbezier(0,0.5)(0.5,0.5)(1,0.5)}%
\def\XX{\NN\HH}
\def\NN{%
    \qbezier(0.5,0)(0.5,0.3)(0.5,0.3)
    \qbezier(0.5,1)(0.5,0.7)(0.5,0.7)}%
\def\BPDfr##1{%
\begin{picture}(1,1)%
    \linethickness{0.08\unitlength}
    ##1
    \thinlines
    \color{lightgray}%
    \put(0,0){\line(0,1){1}}%
    \put(1,0){\line(0,1){1}}%
    \put(0,0){\line(1,0){1}}%
    \put(0,1){\line(1,0){1}}%
\end{picture}}
\providecommand{\OO}[1][]{
\begin{picture}(1,1)%
    \put(0,0.2){\makebox[\unitlength]{\(##1\)}}
    \color{lightgray}%
    \put(0,0){\line(0,1){1}}%
    \put(1,0){\line(0,1){1}}%
    \put(0,0){\line(1,0){1}}%
    \put(0,1){\line(1,0){1}}%
\end{picture}}
\let\O\OO%
\def\X{\BPDfr{\XX}}
\def\BX{\BPDfr{\II\HH}}
\def\F{\BPDfr{\FF}}
\def\J{\BPDfr{\JJ}}
\def\H{\BPDfr{\HH}}
\def\I{\BPDfr{\II}}
\def\B{\BPDfr{\FF\JJ}}
\def\M##1{\begin{picture}(1,1)%
    \put(0,0.2){\makebox[\unitlength]{\(##1\)}}
\end{picture}}
\def\+{\O[\color{gray}\textnormal{}]}
\begin{array}{@{\,}c@{\,}}
{\def\arraystretch{0}
\setlength{\arraycolsep}{0pc}
\color{teal}
\begin{array}{@{}l@{}}%
#2\end{array}}
\end{array}}

\definecolor{lightcyan}{rgb}{0.8,1,1}

\allowdisplaybreaks

\begin{document}

\title{Bumpless pipe dreams meet Puzzles}
\author{Neil J.Y. Fan}
\address[Neil J.Y. Fan]{Department of Mathematics, 
Sichuan University, Chengdu, Sichuan 610065, P.R. China}
\email{fan@scu.edu.cn}

\author{Peter L. Guo}
\address[Peter L. Guo]{Center for Combinatorics, LPMC, 
Nankai University, Tianjin 300071, P.R. China}
\email{lguo@nankai.edu.cn}

\author{Rui Xiong}
\address[Rui Xiong]{Department of Mathematics and Statistics, University of Ottawa, 150 Louis-Pasteur, Ottawa, ON, K1N 6N5, Canada}
\email{rxion043@uottawa.ca}

\maketitle

\begin{abstract}
Knutson and Zinn-Justin recently found a puzzle rule for the expansion of the product $\G_{u}(x,t)\cdot \G_{v}(x,t)$ of two double Grothendieck polynomials indexed by  permutations with separated descents. 
We establish its  triple Schubert calculus version in the sense of Knutson and Tao, namely, a formula for expanding  $\G_{u}(x,y)\cdot \G_{v}(x,t)$ in different secondary variables. 
Our rule is formulated in terms of  pipe puzzles,  incorporating both the structures of bumpless pipe dreams and classical puzzles. 
  As direct  applications, we recover the separated-descent puzzle formula by Knutson and Zinn-Justin 
 (by setting $y=t$) and the bumpless pipe dream model of double Grothendieck polynomials by 
Weigandt (by setting $v=\operatorname{id}$ and $x=t$). 
Moreover, we utilize the formula to partially confirm a positivity conjecture of Kirillov about applying a skew  operator to a  Schubert polynomial.  
\end{abstract}
\setcounter{tocdepth}{1}
\tableofcontents

\section{Introduction}

The core of this paper is to provide a combinatorial rule for the \emph{triple} Schubert calculus   in the torus-equivariant K-theory  of   flag manifolds, with respect to the basis of  structure sheaves indexed by permutations with separated descents. The geometry of triple Schubert calculus (particularly for the case of cohomology of Grassmannians) was revealed by Knutson and Tao \cite{KT03}. Combinatorially, we shall  give a   formula for expanding the product of two double Grothendieck polynomials in different secondary variables  
\begin{equation}\label{eq:cuvwtoind}
    \G_{u}(x,y)\cdot \G_{v}(x,t)=\sum_w c_{u,v}^w(t,y)\cdot \G_{w}(x,t),
\end{equation}
for two permutations $u$ and $v$ of $\{1,2,\ldots, n\}$ with separated descents
at position $k$, that is, 
\begin{equation}\label{eq:sepdescent}
\mathrm{maxdes}(u)\leq k\leq \mathrm{mindes}(v),
\end{equation}
where $\mathrm{maxdes}(u)=\max\{i\colon u(i)>u(i+1)\}$ and $\mathrm{mindes}(v)=\min\{i\colon v(i)>v(i+1)\}$. 
Here, for the identity permutation $\operatorname{id}=12\cdots n$,  we use the convention that $\mathrm{maxdes}(\operatorname{id})=0$ and $\mathrm{mindes}(\operatorname{id})=+\infty$.

Our  formula for $c_{u,v}^w(t,y)$, see Theorem \ref{thm:mainTh}, is described in terms of ``\emph{pipe puzzles}'', see Section \ref{MMMM} for the precise definition. Pipe puzzles enjoy the features of both bumpless pipe dreams and puzzles, as will be explained  in Section \ref{TGBF}. 
The formula includes the following specializations and applications.
\begin{itemize}
\setlength{\itemsep}{1ex}

    \item[(i)] The case $y=t$. Theorem \ref{thm:mainTh} recovers  the puzzle rule for  permutations with separated descents  by  Knutson and Zinn-Justin \cite[Theorem 1]{Pzz3}, which is manifestly positive in the sense of Anderson,   Griffeth and Miller \cite{AGM} (an equivariant  K-theory extension of  Graham's positivity theorem \cite{Gr01}).  

   \item[(ii)] The case $\beta=0$. Theorem \ref{thm:mainTh} becomes a combinatorial rule for the expansion of the product $\S_{u}(x,y)\cdot \S_{v}(x,t)$ of two Schubert polynomials in different secondary variables, see Theorem \ref{HHUU}. We point out that in the case    $y=t=0$,  Huang  \cite{Huang21} derived a tableau formula for the product $\S_{u}(x)\cdot \S_{v}(x)$ of two single Schubert polynomials for $u, v$ with separated descents. 

    \item[(iii)]
    The case that both $u$ and $v$ are   $k$-Grassmannian permutations. Theorem \ref{thm:mainTh}  extends the puzzle formula for the product $\G_{\lambda}(x,t)\cdot \G_{\mu}(x,t)$  by  Wheeler and   Zinn-Justin \cite[Theorem 2]{WZ19} (The latter formula on the one hand is an equivariant extension of   Vakil's puzzle formula  \cite{Vakil06} for the product $\G_{\lambda}(x)\cdot \G_{\mu}(x)$ of two single Grothendieck polynomials, and on the other hand is a K-theory extension of the Knutson--Tao puzzle formula \cite{KT03} for the product $s_{\lambda}(x,t)\cdot s_{\mu}(x,t)$ of two double Schur polynomials).
    
    We remark that (1) an alternative puzzle formula (different from the one  in  \cite{WZ19})  for $\G_{\lambda}(x,t)\cdot \G_{\mu}(x,t)$  was conjectured by Knutson and Vakil, and proved by Pechenik and   Yong \cite{PY15} (after a modification), (2) Wheeler and   Zinn-Justin \cite[Theorems 2'' and 3'']{WZ19} gave puzzle formulas for the product of two \emph{dual} Grothendieck polynomials  in different secondary variables, and (3) 
     puzzle formulations  of the Molev--Sagan tableau formula \cite{MS99} for the product $s_{\lambda}(x,y)\cdot s_{\mu}(x,t)$ of two double Schur polynomials  in different secondary variables  were given by  Knutson and Tao \cite[Section 6]{KT03} and Zinn-Justin \cite{ZJ2009}. 

    \item[(iv)] The case that  $k=n$ (this means  $u$ may be any permutation of $\{1,2,\ldots, n\}$), $v=\operatorname{id}$, and $x=t$. Theorem \ref{thm:mainTh} reduces to the bumpless pipe dream model of double Grothendieck polynomials by Weigandt \cite{Weigandt21}, which, by setting $\beta=0$, leads to the bumpless pipe dream model of   double Schubert polynomials    by Lam, Lee and Shimozono \cite{LLS21}.
    An alternative   proof of Weigandt's model  
was given by Buciumas and  Scrimshaw \cite{BS22} based on colored lattice models.

    \item[(v)]
Kirillov \cite[Conjecture 1]{Kirillov07} conjectured that after applying the skew  divided difference operator 
$\partial_{w/v}$ to a single Schubert polynomial $\S_u(x)$,  the resulting polynomial will have   nonnegative integer coefficients. Setting $y=0$ in Theorem \ref{HHUU} leads to a positive expansion of the product $\S_u(x)\cdot\S_v(x,t)$,  which, as will be explained in Section \ref{KiriCon},  allows us to confirm  Kirillov's prediction   for $u$ and $v$ with separated descents   and arbitrary $w$.     
     
\end{itemize}

An innovation  in our approach is finding   that   $c_{u,v}^w(t,y)$ satisfies two kinds of   recurrence relations, as given in Section \ref{RCR}. 
When $u$ and $v$ have separated descents, such recurrence relations, together with an initial condition, fully determine the computation of $c_{u,v}^w(t,y)$. This could essentially simplify the proof of Theorem  \ref{thm:mainTh}. Specifically, we may 
show that our pipe puzzle formula enjoys the same recurrence relations and initial condition (without too much efforts) by realizing pipe puzzles as an integrable  lattice model.
 We remark that the above mentioned recurrence relations are no longer  { available} in the case $y=t$.  This  means  in some sense that while the problem of computing triple Schubert structure constants is broader, 
its proof could be  simpler. 

 This paper is arranged as follows.  In Section \ref{MMMM}, we state the pipe puzzle  formula for $c_{u,v}^w(t,y)$ in the case that $u, v$  are permutations with separated descents, see Theorem  \ref{thm:mainTh}. In   Section \ref{RCR},  we provide two recurrence relations for  $c_{u,v}^w(t,y)$, and explain that such recurrence relations still work when restricted to  permutations with separated descents. 
 In Section \ref{LLMM}, we realize pipe puzzles as a lattice model, and show that it satisfies two types of  Yang--Baxter equations.  In Section \ref{EFBG}, based on the lattice model, we show that our pipe puzzle formula 
 satisfies the same recurrence relations  as $c_{u,v}^w(t,y)$,   thus completing the proof of Theorem  \ref{thm:mainTh}. 
 Section \ref{TGBF} is devoted to applications of Theorem  \ref{thm:mainTh}, mainly including those aforementioned. 

\subsection*{Ackowledgement}
We are grateful to Paul Zinn-Justin for valuable discussions and suggestions. Parts of this work were completed while the authors participated in the program ``PKU Algebra and Combinatorics Experience'' held at  Beijing International Center for Mathematical Research, Peking University, and we wish to thank Yibo Gao for the invitation and hospitality.
This work was supported by the National Natural Science Foundation of China (11971250, 12071320, 12371329).
R.X. acknowledges the partial support from the NSERC Discovery grant RGPIN-2015-04469, Canada. 

\section{Main result}\label{MMMM}

The main result is  given in  Theorem  \ref{thm:mainTh}, a separated-descent  pipe puzzle formula  for the coefficients $c_{u,v}^w(t,y)$.
Let us  begin  by giving the definition of    Grothendieck polynomials.
As usual, we use $S_n$ to denote the symmetric  group of permutations of $\{1,2,\ldots, n\}$. 
Let $\beta$ be a formal variable. 
Denote 
$$
x\ominus y = \frac{x-y}{1+\beta y}.$$
Let  $\pi_i$ be the \emph{Demazure operator}:
$$\pi_i f = \frac{(1+\beta x_{i+1})f-(1+\beta x_i)f|_{x_i\leftrightarrow x_{i+1}}}{x_i-x_{i+1}}.$$
The \emph{double Grothendieck polynomial} $\G_w(x, t)$ for $w\in S_\infty=\bigcup_{n\geq 0 }S_n$ is determined by the following two properties:
$$\G_{n\cdots 21}(x,t)=\prod_{i+j\leq n}(x_i\ominus t_j);$$
$$\pi_i\G_w(x,t)=\G_{ws_i}(x,t),\qquad \text{if }w(i)>w(i+1).$$
Here, $s_i=(i, i+1)$ is the simple transposition, and $ws_i$ is obtained from $w$ by swapping $w(i)$
and $w(i+1)$. 
Since $\pi_i^2=-\beta\pi_i$, it follows that 
\begin{align}\label{eq:piiG}
\pi_i\G_{w}(x,t)&=\begin{cases}
    \G_{ws_i}(x,t), & \text{if }w(i)>w(i+1),\\[5pt]
    -\beta\G_{w}(x,t), & \text{if }w(i)<w(i+1).
\end{cases}
\end{align}
Letting $t_i=0$ defines  the \emph{single Grothendieck polynomial} 
$$\G_w(x)=\G_w(x,0).$$
Setting $\beta=0$, we get the \emph{double (resp., single) Schubert polynomial} 
$$\S_w(x,t)=\G_w(x,t)|_{\beta=0},\qquad (\text{resp., } \S_w(x)=\G_w(x)|_{\beta=0}).$$

\begin{Rmk}\label{rmk:chavar}
There appear different   definitions for Grothendieck polynomials     in the literature, which will be equivalent after appropriate changes of variables. For example,  \cite{Pzz3} adopts the following   operator and initial condition:  
\def\X{X}\def\T{T}
$$\bar{\partial}_if=\frac{\X_{i+1}f-\X_if|_{\X_i\leftrightarrow \X_{i+1}}}{\X_{i+1}-\X_{i}},\qquad 
\mathcal{G}_{n\cdots 21}(\X,\T)=\prod_{i+j\leq n}\left(1-\X_i/\T_j\right).$$
It can be  checked that  $\mathcal{G}_{w}(\X,\T)$ can be obtained from ${\G}_{w}(x,t)$ by the following  replacements:
$$
\beta=-1,\qquad \X_i=1-x_i, \qquad 
\T_i=1-t_i.$$ 
Our definition is consistent with that used  in \cite[Section 5.1]{LLS22}.
\end{Rmk}

In the remaining of this section, we assume that  $u$ and $v$ are permutations of $S_n$ with separated descents at position $k$.  We are going to describe  our pipe puzzle formula for $c_{u,v}^w(t,y)$.  To begin, consider an $n$ by $n$ grid with   labeled  boundary:  
\begin{align}
\label{eq:boarduvw}
\BPD[1.5pc]{
\M{}\M{\theta_{v}^1}\M{\theta_v^2}\M{\cdots}\M{\cdots}\M{\theta_v^n}\\
\M{0}\+\+\M{\cdots}\M{\cdots}\+\M{\kappa_{u}^1}\\
\M{0}\+\+\M{\cdots}\M{\cdots}\+\M{\kappa_{u}^2}\\
\M{\vdots}\M{\vdots}\M{\vdots}\M{\ddots}\M{\ddots}\M{\vdots}\M{\vdots}\\
\M{\vdots}\M{\vdots}\M{\vdots}\M{\ddots}\M{\ddots}\M{\vdots}\M{\vdots}\\
\M{0}\+\+\M{\cdots}\M{\cdots}\+\M{\kappa_{u}^n}\\
\M{}\M{\eta_w^1}\M{\eta_w^2}\M{\cdots}\M{\cdots}\M{\eta_w^n}
}
\qquad 
\begin{array}{r@{\,}l}
\kappa_u^i &=\begin{cases}
    u^{-1}(i), & u^{-1}(i) \leq k,\\[5pt]
    0, & u^{-1}(i)>k.
    \end{cases}\\[4ex]
\theta_v^i &=\begin{cases}
    0, & v^{-1}(i)\leq k,\\[5pt]
    v^{-1}(i), & v^{-1}(i) >k.
    \end{cases}\\[4ex]
\eta_w^i & = w^{-1}(i).
\end{array}
\end{align}

    

 
We see that  the  nonzero labels on 
the right side are $1, \ldots, k$, and the nonzero labels on the top side are $k+1,\ldots, n$. There is  no obstruction to rebuilding  $u$, $v$ and $w$ from the boundary labeling, because of the separated-descent assumption.   For the sake of  brevity, the label 0 on the boundary will often  be omitted. See   Example \ref{eg:Schupipepuzzle} for the boundary labeling for  $u=42135$, $v=14532$,  $w=53412$, and $k=2$.

Our formula is a weighted counting of tilings of the $n$ by $n$ grid by unit tiles 
(with pipes), subject to certain conditions. 
To warm up, we   first give the formula for double Schubert polynomials.

\subsection{Statement for double Schubert polynomials} 
Assume that
\begin{equation}\label{QQSSCC}
    \S_{u}(x,y)\cdot \S_{v}(x,t)=\sum_{w}\overline{c}_{u,v}^w(t,y) \cdot \S_{w}(x,t).
\end{equation}
The admissible tiles are  
\begin{align}\label{eq:Schtile}
\begin{array}{cccccc}
\BPD[1.5pc]{\O}&
\BPD[1.5pc]{\X}&
\BPD[1.5pc]{\F}&
\BPD[1.5pc]{\J}&
\BPD[1.5pc]{\I}&
\BPD[1.5pc]{\H}\\
\end{array}
\end{align}
The curves drawn  on the tiles are referred to as \emph{pipes}. 
A tiling of \eqref{eq:boarduvw} built upon the tiles in \eqref{eq:Schtile} is a network of pipes such that 
\begin{enumerate}
    \item there are a total of $n$ pipes, among which $k$ pipes   enter horizontally from rows   on the right side labeled $1,\ldots, k$, and $n-k$ pipes enter vertically  from   columns  on the top side labeled $k+1,\ldots, n$. 
    The  pipes inherit the labels of the corresponding rows and columns. 

    \item the $n$ pipes end vertically  on the bottom side, such that the label of each pipe matches the label of the column where it ends.  
\end{enumerate}

A \emph{Schubert pipe puzzle} for $u,v,w$ is a tiling of \eqref{eq:boarduvw} with the tiles in \eqref{eq:Schtile}, subject to the following restriction on the tiles $\BPD{\X}$:
\begin{equation}\label{eq:crossRes}
\fbox{\begin{minipage}{0.75\linewidth}
The horizontal pipe in $\BPD{\X}$ must receive a smaller label. 
For example, 
    $$
    \BPD{\M{3}\M{}\M{}\\
    \I\O\O\\
    \I\O\F\M{2}\\
    \I\F\X\M{1}\\
    \M{3}\M{1}\M{2}}\text{ is allowed, while }
    \BPD{\M{3}\M{}\M{}\\
    \I\O\O\\
    \I\O\F\M{1}\\
    \I\F\BPDfr{\color{red}\XX}\M{2}\\
    \M{3}\M{2}\M{1}}\text{ is not allowed.}$$
\end{minipage}}
\end{equation}
Denote by $\PP_0(u,v,w)$ the set of Schubert pipe puzzles for $u,v,w$. 
For each $\pi \in \PP_0(u,v,w)$, define its \emph{Schubert weight} by
$$\wt_0(\pi)=\prod_{(i,j)} (t_j-y_i),$$
where the sum is over empty tiles $\BPD{\O}$ at the $(i,j)$-positions (in the matrix coordinate).

\begin{Th}\label{HHUU}
Let  $u,v \in S_n$ be permutations with separated descents   at position $k$.
For $w\in S_n$, we have
\begin{equation}\label{POIU}
    \overline{c}_{u, v}^w(t,y) = \sum_{\pi\in \PP_0(u,v,w)} \wt_0(\pi).
\end{equation} 
\end{Th}

\begin{Rmk}\label{TREWW}
It may happen that $\S_{w}(x, t)$, $w\in S_{n'}$ with $n<n'$, appears in the expansion of $\S_{u}(x,y)\cdot \S_{v}(x,t)$. In such a case,  to compute $\overline{c}_{u, v}^w(t,y)$, one needs only to embed naturally  $ S_n$ into $S_{n'}$, and then apply Theorem \ref{HHUU} ($u$ and $v$ are now viewed as permutations in $S_{n'}$).     
\end{Rmk}


\begin{Eg}\label{eg:Schupipepuzzle}
Let $u=42135$, $v=14532$, and set $k=2$.  For $w=53412$, there are four Schubert pipe puzzles in $\PP_0(u,v,w)$: 
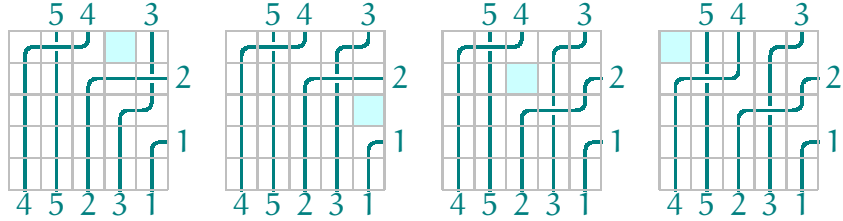
$$\def\o{\BPDfr{\put(0,0){\color{lightcyan}\rule{\unitlength}{\unitlength}}}}
\BPD{\M{}\M{5}\M{4}\M{}\M{3}\\
\F\X\J\o\I\M{}\\
\I\I\F\H\X\M{2}\\
\I\I\I\F\J\M{}\\
\I\I\I\I\F\M{1}\\
\I\I\I\I\I\M{}\\
\M{4}\M{5}\M{2}\M{3}\M{1}}\ \ 
\BPD{\M{}\M{5}\M{4}\M{}\M{3}\\
\F\X\J\F\J\M{}\\
\I\I\F\X\H\M{2}\\
\I\I\I\I\o\M{}\\
\I\I\I\I\F\M{1}\\
\I\I\I\I\I\M{}\\
\M{4}\M{5}\M{2}\M{3}\M{1}}\ \ 
\BPD{\M{}\M{5}\M{4}\M{}\M{3}\\
\F\X\J\F\J\M{}\\
\I\I\o\I\F\M{2}\\
\I\I\F\X\J\M{}\\
\I\I\I\I\F\M{1}\\
\I\I\I\I\I\M{}\\
\M{4}\M{5}\M{2}\M{3}\M{1}}\ \ 
\BPD{\M{}\M{5}\M{4}\M{}\M{3}\\
\o\I\I\F\J\M{}\\
\F\X\J\I\F\M{2}\\
\I\I\F\X\J\M{}\\
\I\I\I\I\F\M{1}\\
\I\I\I\I\I\M{}\\
\M{4}\M{5}\M{2}\M{3}\M{1}}
$$
Here, the empty tiles are colored. 
So it follows from \eqref{POIU} that
\begin{align*}
\overline{c}_{42135,14532}^{53412}& 
= (t_4-y_1)+(t_5-y_3)+(t_3-y_2)+(t_1-y_1).
\end{align*}
\end{Eg}

\subsection{Statement for double Grothendieck polynomials}
We allow one more admissible tile than \eqref{eq:Schtile}:
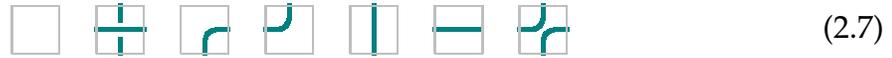
\begin{align}\label{eq:Grotile}
\begin{array}{ccccccc}
\BPD[1.5pc]{\O}&
\BPD[1.5pc]{\X}&
\BPD[1.5pc]{\F}&
\BPD[1.5pc]{\J}&
\BPD[1.5pc]{\I}&
\BPD[1.5pc]{\H}&
\BPD[1.5pc]{\B}\\
\end{array}
\end{align}
The extra tile in \eqref{eq:Grotile} is a ``bumping'' tile $\BPD{\B}$.  
We shall have the following  restrictions  on the usage of $\BPD{\B}$: 
\begin{equation}\label{eq:bumpResA}
\fbox{\begin{minipage}{0.75\linewidth}
If the two pipes in $\BPD{\B}$ are from the same side, then   the northwest pipe    must receive   a larger  label. 
For example, 
    $$
    \BPD{\M{3}\M{}\M{}\\
    \I\O\O\\
    \I\O\F\M{2}\\
    \I\F\B\M{1}\\
    \M{3}\M{2}\M{1}}\text{ is allowed, while } 
    \BPD{\M{3}\M{}\M{}\\
    \I\O\O\\
    \I\O\F\M{1}\\
    \I\F\BPDfr{\color{red}\JJ\FF}\M{2}\\
    \M{3}\M{1}\M{2}}\text{ is not allowed.}$$
\end{minipage}}
\end{equation}
\begin{equation}\label{eq:bumpResB}
\fbox{\begin{minipage}{0.75\linewidth}
If the two pipes in  $\BPD{\B}$ are from different sides, then  the northwest pipe   must  enter from the right side {(equivalently, it receives a smaller label)}. 
For example, 
    $$
    \BPD{\M{}\M{}\M{3}\\
    \O\F\X\M{2}\\
    \F\B\J\\
    \I\I\F\M{1}\\
    \M{2}\M{3}\M{1}}\text{ is allowed, while }
    \BPD{\M{}\M{}\M{3}\\
    \O\F\BPDfr{\color{red}\JJ\FF}\M{2}\\
    \F\X\J\\
    \I\I\F\M{1}\\
    \M{2}\M{3}\M{1}}\text{ is not allowed.}$$
\end{minipage}}
\end{equation}
A \emph{(Grothendieck) pipe puzzle} for $u,v,w$ is a tiling of \eqref{eq:boarduvw} with the tiles in \eqref{eq:Grotile} obeying 
the restriction \eqref{eq:crossRes} on $\BPD{\X}$, as well as 
the restrictions \eqref{eq:bumpResA} and \eqref{eq:bumpResB} on $\BPD{\B}$. 

Let $\PP(u,v,w)$ be the set of pipe puzzles for $u,v,w$. 
For $\pi\in \PP(u,v,w)$, its \emph{weight} $\wt(\pi)$ is the product of  factors contributed by all tiles of $\pi$:   at the $(i,j)$-position,
\begin{enumerate}
    \item an empty tile $\BPD{\O}$  contributes $t_j\ominus y_i$; 
    
    \item  an elbow tile $\BPD{\J}$, in which the  pipe is from the right side,  contributes $1+\beta(t_j\ominus y_i)$; 
    
    \item an elbow tile $\BPD{\F}$, in which the  pipe is from the top side,  contributes $1+\beta(t_j\ominus y_i)$; 
    
    \item a bumping  tile  $\BPD{\B}$, in which the two  pipes are  from the same side,  contributes $\beta$; 
    
    \item  a bumping tile  $\BPD{\B}$, in which the two pipes are from different sides,  contributes $\beta(1+\beta(t_j\ominus y_i))$.  

    \item any other  tile except for the above cases contributes 1.
\end{enumerate}

\begin{Th}\label{thm:mainTh}
Let  $u,v \in S_n$ be permutations with separated descents   at position $k$.
For $w\in S_n$, we have
\begin{equation}\label{PPKYC}
    c_{u,v}^w(t,y) = \sum_{\pi\in \PP(u,v,w)} \wt(\pi).
\end{equation}
\end{Th} 

Note that Remark \ref{TREWW} is still  valid for Theorem \ref{thm:mainTh}.
We also remark that Theorem \ref{thm:mainTh} specializes to Theorem \ref{HHUU} in the case  $\beta=0$  by noticing that  $\wt(\pi)|_{\beta=0}=0$ whenever  $\pi\notin \PP_0(u,v,w)$, and $\wt(\pi)|_{\beta=0}=\wt_0(\pi)$ for $\pi\in \PP_0(u,v,w)$.

\begin{Eg}\label{eg:pipepuzzle}
Take the same setting as in Example \ref{eg:Schupipepuzzle}. There are nine pipe puzzles in $ \PP(u,v,w)$, among which the   pipe puzzles in the top row  are those appearing in Example \ref{eg:Schupipepuzzle}.  Here, the tiles with weights not equal to $1$ are colored. 
$$\def\o{\BPDfr{\put(0,0){\color{lightcyan}\rule{\unitlength}{\unitlength}}}}
\def\f{\BPDfr{\put(0,0){\color{lightcyan}\rule{\unitlength}{\unitlength}}\FF}}
\def\j{\BPDfr{\put(0,0){\color{lightcyan}\rule{\unitlength}{\unitlength}}\JJ}}
\def\b{\BPDfr{\put(0,0){\color{lightcyan}\rule{\unitlength}{\unitlength}}\FF\JJ}}
\def\Bb{\BPDfr{\put(0,0){\color{lightcyan}\rule{\unitlength}{\unitlength}}\FF\JJ}}
\BPD{\M{}\M{5}\M{4}\M{}\M{3}\\
\f\X\J\o\I\M{}\\
\I\I\F\H\X\M{2}\\
\I\I\I\f\J\M{}\\
\I\I\I\I\F\M{1}\\
\I\I\I\I\I\M{}\\
\M{4}\M{5}\M{2}\M{3}\M{1}}
\BPD{\M{}\M{5}\M{4}\M{}\M{3}\\
\f\X\J\f\J\M{}\\
\I\I\F\X\H\M{2}\\
\I\I\I\I\o\M{}\\
\I\I\I\I\F\M{1}\\
\I\I\I\I\I\M{}\\
\M{4}\M{5}\M{2}\M{3}\M{1}}
\BPD{\M{}\M{5}\M{4}\M{}\M{3}\\
\f\X\J\f\J\M{}\\
\I\I\o\I\F\M{2}\\
\I\I\F\X\j\M{}\\
\I\I\I\I\F\M{1}\\
\I\I\I\I\I\M{}\\
\M{4}\M{5}\M{2}\M{3}\M{1}}
\BPD{\M{}\M{5}\M{4}\M{}\M{3}\\
\o\I\I\f\J\M{}\\
\f\X\J\I\F\M{2}\\
\I\I\F\X\j\M{}\\
\I\I\I\I\F\M{1}\\
\I\I\I\I\I\M{}\\
\M{4}\M{5}\M{2}\M{3}\M{1}}
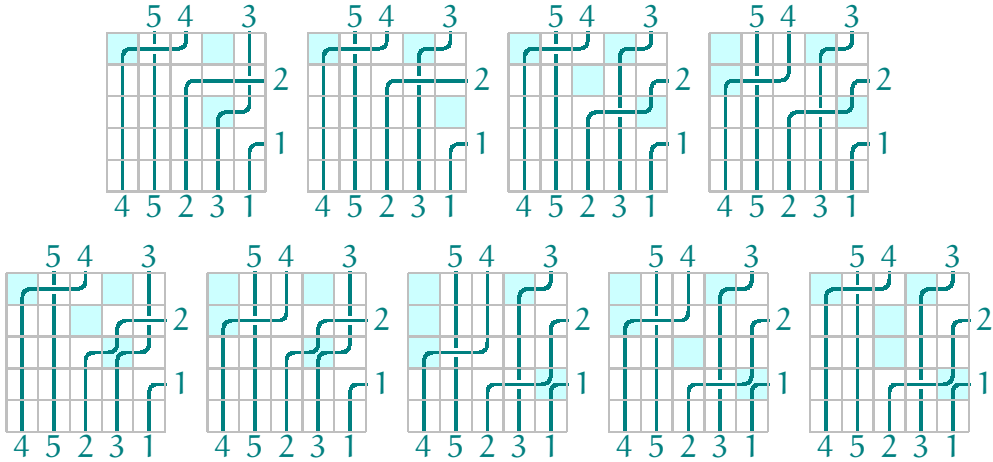
$$
$$\def\o{\BPDfr{\put(0,0){\color{lightcyan}\rule{\unitlength}{\unitlength}}}}
\def\f{\BPDfr{\put(0,0){\color{lightcyan}\rule{\unitlength}{\unitlength}}\FF}}
\def\j{\BPDfr{\put(0,0){\color{lightcyan}\rule{\unitlength}{\unitlength}}\JJ}}
\def\b{\BPDfr{\put(0,0){\color{lightcyan}\rule{\unitlength}{\unitlength}}\FF\JJ}}
\def\Bb{\BPDfr{\put(0,0){\color{lightcyan}\rule{\unitlength}{\unitlength}}\FF\JJ}}
\BPD{\M{}\M{5}\M{4}\M{}\M{3}\\
\f\X\J\o\I\M{}\\
\I\I\o\F\X\M{2}\\
\I\I\F\b\J\M{}\\
\I\I\I\I\F\M{1}\\
\I\I\I\I\I\M{}\\
\M{4}\M{5}\M{2}\M{3}\M{1}}
\BPD{\M{}\M{5}\M{4}\M{}\M{3}\\
\o\I\I\o\I\M{}\\
\f\X\J\F\X\M{2}\\
\I\I\F\b\J\M{}\\
\I\I\I\I\F\M{1}\\
\I\I\I\I\I\M{}\\
\M{4}\M{5}\M{2}\M{3}\M{1}}
\BPD{\M{}\M{5}\M{4}\M{}\M{3}\\
\o\I\I\f\J\M{}\\
\o\I\I\I\F\M{2}\\
\f\X\J\I\I\M{}\\
\I\I\F\X\Bb\M{1}\\
\I\I\I\I\I\M{}\\
\M{4}\M{5}\M{2}\M{3}\M{1}}
\BPD{\M{}\M{5}\M{4}\M{}\M{3}\\
\o\I\I\f\J\M{}\\
\f\X\J\I\F\M{2}\\
\I\I\o\I\I\M{}\\
\I\I\F\X\Bb\M{1}\\
\I\I\I\I\I\M{}\\
\M{4}\M{5}\M{2}\M{3}\M{1}}
\BPD{\M{}\M{5}\M{4}\M{}\M{3}\\
\f\X\J\f\J\M{}\\
\I\I\o\I\F\M{2}\\
\I\I\o\I\I\M{}\\
\I\I\F\X\Bb\M{1}\\
\I\I\I\I\I\M{}\\
\M{4}\M{5}\M{2}\M{3}\M{1}}
$$

As a result, 
\begin{align*}
{c}_{42135,14532}^{53412}& 
= (t_4\ominus y_1)(1+\beta(t_1\ominus y_1))(1+\beta(t_4\ominus y_3))\\
&\quad +(t_5\ominus y_3)(1+\beta(t_1\ominus y_1))(1+\beta(t_4\ominus y_1))\\
&\quad +(t_3\ominus y_2)(1+\beta(t_1\ominus y_1))(1+\beta(t_4\ominus y_1))(1+\beta(t_5\ominus y_3))\\
&\quad +(t_1\ominus y_1)(1+\beta(t_4\ominus y_1))(1+\beta(t_1\ominus y_2))(1+\beta(t_5\ominus y_3))\\
&\quad +\beta(t_4\ominus y_1)(t_3\ominus y_2)(1+\beta(t_1\ominus y_1))(1+\beta(t_4\ominus y_3))\\
&\quad +\beta(t_1\ominus y_1)(t_4\ominus y_1)(1+\beta(t_1\ominus y_2))(1+\beta(t_4\ominus y_3))\\
&\quad+\beta(t_1\ominus y_1)(t_1\ominus y_2)(1+\beta(t_4\ominus y_1))(1+\beta(t_1\ominus y_3))\\
&\quad + \beta(t_1\ominus y_1)(t_3\ominus y_3)(1+\beta(t_4\ominus y_1))(1+\beta(t_1\ominus y_2))\\
&\quad+ \beta(t_3\ominus y_2)(t_3\ominus y_3)(1+\beta(t_1\ominus y_1))(1+\beta(t_4\ominus y_1)).
\end{align*}
\end{Eg}


\section{Recurrence relations }\label{RCR}

In this section, we present two recurrence relations, as well as an initial condition, for $c_{u,v}^w(t,y)$, and explain that they can be used to determine the computation of $c_{u,v}^w(t,y)$ for $u, v\in S_n$ with separated descents.

Let us first review  the definition of Bruhat order on $S_n$. Let  $t_{ij}$ ($1\leq i<j\leq n$) denote the transpositions of $S_n$. Then $S_n$ is generated by the set of simple transpositions $s_i=t_{i\, i+1}$ for $1\leq i<n$. The length $\ell(w)$ of $w\in S_n$
 is the minimum number of simple transpositions appearing in any decomposition $w=s_{i_1}\cdots s_{i_m}$.
 It is well known that $\ell(w)$ equals the number of inversions of $w$:
 \[\ell(w)=\#\{(i, j)\colon 1\leq i<j\leq n,\, w(i)>w(j)\}.\]
 
Notice that $wt_{ij}$ (resp., $t_{ij}w$) is obtained from $w$ by swapping $w(i)$ and $w(j)$ (resp., the values $i$ and $j$). 
Write  $w<wt_{ij}$ if $\ell(w)<\ell(wt_{ij})$ (namely, $w(i)<w(j)$).  The transitive closure
of all relations $w<wt_{ij}$  forms the Bruhat order $\leq $ on $S_n$. 
It should be noted that the Bruhat order can be defined equivalently as the transitive closure of relations $w<t_{ij} w$ (which means $\ell(w)<\ell(t_{ij}w)$).

In the rest of this section, we shall often encounter the situation $s_iw<w$ or $s_iw>w$.  By definition, $s_iw<w$ means $i$ appears after $i+1$ in $w$, while $s_iw>w$ means $i$ appears before $i+1$ in $w$. 

The two recurrence relations for $c_{u,v}^w(t,y)$ can be stated as follows.  If there is no confusion occurring, we sometimes simply write $c_{u,v}^w$ for $c_{u,v}^w(t,y)$. 

\begin{Prop}\label{indonu}
 If $s_iu<u$, then 
\begin{equation}
    c_{s_iu,v}^w
    =-\frac{1+\beta y_i}{y_i-y_{i+1}}c_{u, v}^w
    +\frac{1+\beta y_{i+1}}{y_i-y_{i+1}}c_{u, v}^w|_{y_i\leftrightarrow y_{i+1}}.
\end{equation} 
\end{Prop}

\begin{Prop}\label{indonw}
If $s_iw>w$, then 
\begin{align}\label{PURW-1}
c_{u, v}^{s_iw} & 
=\begin{cases}
  -\dfrac{1+\beta t_{i+1}}{t_i-t_{i+1}}c_{u, v}^w|_{t_i\leftrightarrow t_{i+1}}
+\dfrac{1+\beta t_{i}}{t_i-t_{i+1}}c_{u, v}^w
+ c_{u,s_iv}^w|_{t_i\leftrightarrow t_{i+1}}, 
& s_iv<v,\\[15pt]
  -\dfrac{1+\beta t_{i}}{t_i-t_{i+1}}c_{u, v}^w|_{t_i\leftrightarrow t_{i+1}}
+\dfrac{1+\beta t_{i}}{t_i-t_{i+1}}c_{u, v}^w,
& s_iv>v. 
\end{cases}
\end{align}
\end{Prop}

In order to more quickly go into the proof of Theorem  \ref{thm:mainTh}, we put the proofs of Propositions \ref{indonu} 
and \ref{indonw}  in Appendix \ref{sec:indform}.  

To give the initial condition, we need the following   localization 
\begin{equation}\label{eq:locofG}
    \G_w(t,t)=\begin{cases}
        1, & w =\operatorname{id},\\ 
        0, & \text{otherwise}.
    \end{cases}
\end{equation}
This is the very special case of the general localization formula for Grothendieck polynomials, 
see for example Buch and  Rim\'{a}nyi \cite{BR} and the references therein. 
Taking $x=t$ in \eqref{eq:cuvwtoind} and then applying \eqref{eq:locofG}, we obtain  the following relationship.

\begin{Lemma}\label{Prop:cuvid=G}
We have 
\begin{equation*}
    c_{u, v}^{\operatorname{id}}(t,y)
    = \begin{cases}
        \G_u(t,y), & \text{if }v=\operatorname{id},\\[5pt]
        0, & \text{otherwise}.
    \end{cases}
\end{equation*}
\end{Lemma}

Denote by $u_0=n(n{-}1)\cdots (n{-}k{+}1)\,1 2\cdots (n{-}k)\in S_n$ the unique longest permutation among those  $u\in S_n$ with
$\mathrm{maxdes}(u)\leq k$:
\begin{equation}\label{eq:defu0}
    u_0(i)=\begin{cases}
n+1-i, & i\leq k,\\ 
i-k,   & k<i\leq n.
\end{cases}
\end{equation}
By direct computation, we have 
$$\G_{u_0}(x, t)=\prod_{i=1}^k \prod_{j=1}^{n-i}(x_i\ominus t_j).$$
Actually, this is clearly true for $k=n$.
If the statement is true for $k$, then  applying operators $\pi_{1}\cdots \pi_{k}$, we can compute  the case of $k-1$. 
This, along with Lemma \ref{Prop:cuvid=G}, leads to the  initial condition. 

\begin{Prop}\label{Prop:initial}
For $v\in S_n$,  
\begin{equation*}
    c_{u_0,v}^{\operatorname{id}}
    = \begin{cases}\displaystyle
        \prod_{i=1}^k\prod_{j=1}^{n-i}(t_i\ominus y_j), & \text{if }v=\operatorname{id},\\[5pt]
        0, & \text{otherwise}.
    \end{cases}
\end{equation*}
\end{Prop}

Propositions \ref{indonu} and \ref{indonw}
are valid for any $u, v, w\in S_n$. We  explain that such recurrences are  closed when restricting $u, v\in S_n$ 
to permutations with separated descents at $k$. 
In other words, we could use    Propositions \ref{indonu} and \ref{indonw} (only applied to permutations with separated descents at $k$), along with  the initial condition in Proposition  \ref{Prop:initial}, to  compute  $c_{u,v}^w(t,y)$ for  any $u, v \in S_n$ with separated descents at $k$. 

\begin{itemize}

\setlength{\itemsep}{1ex}

    \item First, compute $c_{u,v}^{\mathrm{id}}(t,y)$ for $w=\mathrm{id}$. The initial case is for the longest permutation $u=u_0$, as  done  in Proposition \ref{Prop:initial}. We next consider $c_{u,v}^{\mathrm{id}}(t,y)$ with $\ell(u)<\ell(u_0)$. Since $u\neq u_0$, one can always choose an integer   $i$ among the first $k$ values $u(1),\ldots, u(k)$, such that $i$ appears before $i+1$ in $u$. For example, given  $u=7423156$ and $k=4$,
   we may choose $i=4$ or $i=2$. 
    
    Now  we have $u< s_iu\in S_n$. It is easily checked that  $\mathrm{maxdes}(s_iu)\leq k$. 
    Set $u'=s_iu$. By induction on the length of $u$, the value of  $c_{u',v}^{\mathrm{id}}(t,y)$ is known, which allows us to compute  $ c_{u,v}^{\mathrm{id}}(t,y)=c_{s_iu',v}^{\mathrm{id}}(t,y)$ from   $c_{u',v}^{\mathrm{id}}(t,y)$
    by means of  Proposition \ref{indonu}.

 \item Second, compute $c_{u,v}^{w}(t,y)$ for $\ell(w)>0$. In this case, choose any $s_i$ such that $s_iw<w$.  It is also easily checked that if $s_iv<v$, then we still have $\mathrm{mindes}(s_i v)\geq k$.   Set $w'=s_iw$.  By induction on the length of $w$, the values of $c_{u,v}^{w'}(t,y)$ and $c_{u,s_iv}^{w'}(t,y)$ are  known. Applying Proposition \ref{indonw}, we may deduce  $c_{u,v}^{w}(t,y)=c_{u,v}^{s_iw'}(t,y)$ from $c_{u,v}^{w'}(t,y)$ and $c_{u,s_iv}^{w'}(t,y)$.
\end{itemize}

\section{Integrable  lattice  models}\label{LLMM}

Throughout this section, we assume that $u, v\in S_n$ are permutations with separated descents at $k$, and $w$ is any permutation in $S_n$. 
We  shall realize the pipe puzzles in  $\PP(u,v,w)$ as a (colored) lattice model, denoted  $\mathrm{L}(u,v,w)$, so that the  right-hand side of 
\eqref{POIU} is equal to the partition function of $\mathrm{L}(u,v,w)$. For more background about lattice models, we refer the reader to, for example,  \cite{BBBG21, BS22, BTW20, ZJ09}. We verify  that $L(u,v,w)$ is  integrable, in the sense  that  it satisfies Yang--Baxter equations with respect to particular choices of $R$-matrices.  

\subsection{Lattice model}
Consider a square grid with $n$ horizontal lines and $n$ vertical lines. The intersection point  of  two lines  will be a vertex (so there are a total of $n^2$ vertices).  
The lines  between two vertices are called edges. We shall also attach additional half edges to the vertices on the boundary, so that there are four half edges around each vertex.  

A \emph{state} is a labeling of all the (half) edges with labels from $\{0, 1,2\ldots, n\}$, with a fixed boundary condition which is  consistent with that in  \eqref{eq:boarduvw}:
the left half edges are all labeled 0, the right half edges are labeled $\kappa_u^1, \ldots, \kappa_u^n$ from top to bottom,  the top (resp., bottom) half edges are labeled $\theta_v^1, \ldots, \theta_v^n$ (resp., $\eta_w^1, \ldots, \eta_w^n$) from left to right. The label of each (half) edge will be marked with a circle, and a vertex will be formally assigned  a parameter $x$.  
A state is \emph{admissible} if the local configurations around each vertex (namely, the labeled half edges adjacent to each vertex) satisfy exactly  one of the conditions as listed  in the middle column of  Table \ref{QQWWEE}.
Moreover, each  allowable local configuration is assigned    a weight as given in the first column  of Table \ref{QQWWEE}.

\def\C#1{\raisebox{0pc}[0.6pc][-0.2pc]{\makebox[0.58pc]{%
\makebox[0pc]{\Large$\bigcirc$}%
\makebox[0pc]{${#1}$}%
}}}
\def\CC#1{\raisebox{-0.3pc}[0.6pc][-0.4pc]{{%
{\setlength{\unitlength}{1.2pc}%
\begin{picture}(1,1)%
\put(0.5,0.5){\oval(1.5,1.1)}
\put(0.5,0.27){\makebox[0pc]{{${#1}$}}}%
\end{picture}%
}}}}
\newcommand{\RmatA}[5][x]{
{\rule[-2.5pc]{0pc}{5.5pc}
\begin{array}{c}
\xymatrix@=0.5pc{
&\C{#2}\\
\C{#4}&#1\ar@{-}[u]\ar@{-}[d]\ar@{-}[l]\ar@{-}[r]&\C{#3}\\
&\C{#5}\\
}
\end{array}
}}
\newcommand{\RmatB}[5][x]{
{\rule[-2.5pc]{0pc}{5.5pc}
\begin{array}{c}
\xymatrix@=0.5pc{
\C{#4}&&\C{#2}\\
&#1\ar@{-}@/_/[ul]\ar@{-}@/^/[ur]\ar@{-}@/_/[dr]\ar@{-}@/^/[dl]\\
\C{#5}&&\C{#3}
}
\end{array}
}}
\newcommand{\RmatC}[5][x]{
{\rule[-2.5pc]{0pc}{5.5pc}
\begin{array}{c}
\xymatrix@=0.5pc{
\C{#2}&&\C{#3}\\
&#1\ar@{-}@/_/[ur]\ar@{-}@/^/[ul]\ar@{-}@/_/[dl]\ar@{-}@/^/[dr]\\
\C{#4}&&\C{#5}
}
\end{array}
}}

\begin{table}[hhhhh]
$$
\def\a{N}\def\b{E}\def\c{W}\def\d{S}
\RmatA{\a}{\b}{\c}{\d} 
:\ \ \ 
\begin{array}{ccc}\hline
\text{weights} & \text{conditions} & \text{tiles}\\\hline 
x &         \a=\b=\c=\d=0 & \BPD{\O} \\[3pt]
1 &         \b=\c=0<\a=\d& \BPD{\I} \\[3pt]
1 &         \a=\d=0<\b=\c& \BPD{\H} \\[3pt]
1 &         0<\b=\c<\a=\d & \BPD{\X}\\[3pt]
1+\beta x & \b=\d=0<\a=\c\leq k & \BPD{\J} \\[3pt]
1 &         \b=\d=0\text{ and }k<\a=\c & \BPD{\J} \\[3pt]
1 &         \a=\c=0<\b=\d\leq k & \BPD{\F} \\[3pt]
1+\beta x&  \a=\c=0\text{ and }k<\b=\d & \BPD{\F} \\[3pt]
\beta &     0<\b=\d<\a=\c\leq k & \BPD{\B}\\[3pt]
\beta &     k<\b=\d<\a=\c & \BPD{\B}\\[3pt]
\beta(1+\beta x) & 0<\a=\c\leq k<\b=\d & \BPD{\B}
\\
\hline
\end{array}
$$
\caption{Weights, local configurations, and tiles.}
\label{QQWWEE}
\end{table}

Each configuration around a vertex naturally corresponds to a tile that is used to define  a pipe puzzle, 
as illustrated in the last column of Table \ref{QQWWEE}, with   pipes inheriting the labels of edges. We display the information in Table \ref{QQWWEE} more intuitively in Table \ref{tab:Lmat}. Therefore, each admissible state generates a pipe puzzle, and vice versa. See Figure \ref{fig:enter-label-YYY} for   an  admissible state and its corresponding pipe puzzle.

\begin{table}[hhhhh]
$$\begin{array}{@{}c@{}|@{}c@{}|@{}c@{}}\hline
    \rule[-2.5pc]{0pc}{5.5pc}
    \RmatA[\BPD{\O}]{0}{0}{0}{0}&
    \RmatA[\BPD{\X}]{A}{a}{a}{A}&
    \RmatA[\BPD{\I}]{a}{0}{0}{a}
    \RmatA[\BPD{\H}]{0}{a}{a}{0}
    \\
\hline\vphantom{\dfrac{1}{2}}
    x & 
    1\quad (0<a<A) &
    1\quad (0<a) \\\hline
\hline
    \rule[-2.5pc]{0pc}{5.5pc}
    \RmatA[\BPD{\J}]{a}{0}{a}{0}&
    \RmatA[\BPD{\F}]{0}{a}{0}{a}&
    \RmatA[\BPD{\B}]{A}{a}{A}{a}
    \\
\hline\vphantom{\dfrac{1}{2}}
    \begin{array}{rl}
    1+\beta x & (0<a\leq k)\\
        1 & (k<a) 
    \end{array} & 
    \begin{array}{rl}
    1 & (0<a\leq k)\\
    1+\beta x & (k<a) 
    \end{array}&
    \begin{array}{rl}
    \beta & (0<a<A\leq k)\\
    \beta & (k<a<A)\\
    \beta(1+\beta x) & (0<A\leq k<a) 
    \end{array}
    \\\hline
\end{array}$$
\caption{Diagram illustration of Table \ref{QQWWEE}.}
\label{tab:Lmat}
\end{table}

\begin{figure}[ht]
    \centering
$$\def\C#1{\raisebox{0pc}[0.4pc][-0.1pc]{\makebox[0.58pc]{%
\makebox[0pc]{$\bigcirc$}%
\makebox[0pc]{${\scriptstyle #1}$}%
}}}
\def\o{\BPDfr{\put(0,0){\color{lightcyan}\rule{\unitlength}{\unitlength}}}}
\def\f{\BPDfr{\put(0,0){\color{lightcyan}\rule{\unitlength}{\unitlength}}\FF}}
\def\j{\BPDfr{\put(0,0){\color{lightcyan}\rule{\unitlength}{\unitlength}}\JJ}}
\def\b{\BPDfr{\put(0,0){\color{lightcyan}\rule{\unitlength}{\unitlength}}\FF\JJ}}
\def\Bb{\BPDfr{\put(0,0){\color{lightcyan}\rule{\unitlength}{\unitlength}}\FF\JJ}}\BPD{\M{}\M{5}\M{4}\M{}\M{3}\\
\o\I\I\f\J\M{}\\
\f\X\J\I\F\M{2}\\
\I\I\o\I\I\M{}\\
\I\I\F\X\Bb\M{1}\\
\I\I\I\I\I\M{}\\
\M{4}\M{5}\M{2}\M{3}\M{1}}\qquad 
\begin{array}{c}
\def\ww#1{\makebox[1.5pc]{-\mbox{\(\scriptstyle #1\)}-}\ar@{-}[u]\ar@{-}[d]}
\xymatrix@R=0.3pc@C=0pc{
&\C{0}&&\C{5}&&\C{4}&&\C{0}&&\C{3}&\\
\C{0}&
\ww{t_1\ominus y_1}&\C{0}&
\ww{t_2\ominus y_1}&\C{0}&
\ww{t_3\ominus y_1}&\C{0}&
\ww{t_4\ominus y_1}&\C{3}&
\ww{t_5\ominus y_1}&\C{0}\\
&\C{0}&&\C{5}&&\C{4}&&\C{3}&&\C{0}&\\
\C{0}&
\ww{t_1\ominus y_2}&\C{4}&
\ww{t_2\ominus y_2}&\C{4}&
\ww{t_3\ominus y_2}&\C{0}&
\ww{t_4\ominus y_2}&\C{0}&
\ww{t_5\ominus y_2}&\C{2}\\
&\C{4}&&\C{5}&&\C{0}&&\C{3}&&\C{2}&\\
\C{0}&
\ww{t_1\ominus y_3}&\C{0}&
\ww{t_2\ominus y_3}&\C{0}&
\ww{t_3\ominus y_3}&\C{0}&
\ww{t_4\ominus y_3}&\C{0}&
\ww{t_5\ominus y_3}&\C{0}\\
&\C{4}&&\C{5}&&\C{0}&&\C{3}&&\C{2}&\\
\C{0}&
\ww{t_1\ominus y_4}&\C{0}&
\ww{t_2\ominus y_4}&\C{0}&
\ww{t_3\ominus y_4}&\C{2}&
\ww{t_4\ominus y_4}&\C{2}&
\ww{t_5\ominus y_4}&\C{1}\\
&\C{4}&&\C{5}&&\C{2}&&\C{3}&&\C{1}&\\
\C{0}&
\ww{t_1\ominus y_5}&\C{0}&
\ww{t_2\ominus y_5}&\C{0}&
\ww{t_3\ominus y_4}&\C{0}&
\ww{t_4\ominus y_5}&\C{0}&
\ww{t_5\ominus y_4}&\C{0}\\
&\C{4}&&\C{5}&&\C{2}&&\C{3}&&\C{1}&}
\end{array}$$
    \caption{Correspondence between a pipe puzzle and an admissible state.}
    \label{fig:enter-label-YYY}
\end{figure}
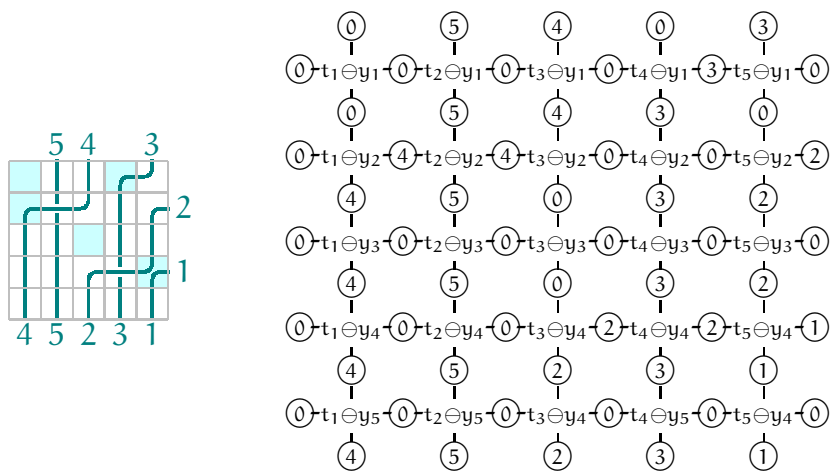

The lattice model  $\mathrm{L}(u, v, w)$ we are considering is defined as the set of all admissible states ($\mathrm{L}(u, v, w)$ can be regarded as a colored lattice model if the labels $1, 2, \ldots, n$ are viewed as $n$ colors).  The weight $\mathrm{wt}(S)$ of a  state $S$ in $\mathrm{L}(u, v, w)$ is the product of all the weights of vertices with $x=t_j\ominus y_i$ in row $i$ and column $j$. The \emph{partition function} of $\mathrm{L}(u, v, w)$ is defined by
\begin{equation*}
Z_{u, v}^w(t,y)=\sum_{S\in \mathrm{L}(u,v,w)} \wt(S). 
\end{equation*}
That is,
\begin{equation*}\label{eq:6vexforduvw}
Z_{u, v}^w(t,y)= 
\begin{array}{c}
\def\objectstyle#1{\scriptstyle #1}
\xymatrix@=0.5pc{
& \CC{\theta_v^1}&&\CC{\theta_v^2} && \cdots && \CC{\theta_v^n}&\\
\C{0}& 
t_1\ominus y_1 \ar@{-}[r]\ar@{-}[l]\ar@{-}[u]\ar@{-}[d]
&\C{}& 
t_2\ominus y_1 \ar@{-}[r]\ar@{-}[l]\ar@{-}[u]\ar@{-}[d]
&\C{}& 
\cdots \ar@{-}[r]\ar@{-}[l]\ar@{-}[u]\ar@{-}[d]
&\C{}& 
t_n\ominus y_1 \ar@{-}[r]\ar@{-}[l]\ar@{-}[u]\ar@{-}[d]
& \CC{\kappa_u^1}\\
&\C{} &&\C{} && &&\C{} &\\
\C{0}& 
t_1\ominus y_2 \ar@{-}[r]\ar@{-}[l]\ar@{-}[u]\ar@{-}[d]
&\C{}&
t_2\ominus y_2 \ar@{-}[r]\ar@{-}[l]\ar@{-}[u]\ar@{-}[d]
&\C{}& 
\cdots \ar@{-}[r]\ar@{-}[l]\ar@{-}[u]\ar@{-}[d]
&\C{}& 
t_n\ominus y_2 \ar@{-}[r]\ar@{-}[l]\ar@{-}[u]\ar@{-}[d]
& \CC{\kappa_u^2}\\
&\C{} &&\C{} && &&\C{} &\\
\vdots & 
\vdots \ar@{-}[r]\ar@{-}[l]\ar@{-}[u]\ar@{-}[d]&& 
\vdots \ar@{-}[r]\ar@{-}[l]\ar@{-}[u]\ar@{-}[d]&& 
\vdots \ar@{-}[r]\ar@{-}[l]\ar@{-}[u]\ar@{-}[d]&& 
\vdots \ar@{-}[r]\ar@{-}[l]\ar@{-}[u]\ar@{-}[d]& \vdots\\
&\C{} &&\C{} && &&\C{} &\\
\C{0}& 
t_1\ominus y_n \ar@{-}[r]\ar@{-}[l]\ar@{-}[u]\ar@{-}[d]
&\C{}& 
t_2\ominus y_n \ar@{-}[r]\ar@{-}[l]\ar@{-}[u]\ar@{-}[d]
&\C{}& 
\cdots \ar@{-}[r]\ar@{-}[l]\ar@{-}[u]\ar@{-}[d]
&\C{}& 
t_n\ominus y_n \ar@{-}[r]\ar@{-}[l]\ar@{-}[u]\ar@{-}[d]
& \CC{\kappa_u^n}\\
&\CC{\eta_w^1}&&\CC{\eta_w^2}&&\cdots&&\CC{\eta_w^n}&}
\end{array}
\end{equation*}

\begin{Rmk}
    It can be checked directly  from Table \ref{tab:Lmat} that the weights of vertices (with $x=t_j\ominus y_i$ in row $i$ and column $j$)
are consistent with the weights of 
the corresponding tiles as defined  above   Theorem \ref{thm:mainTh}.  
\end{Rmk}

Collecting the above observations, we summarize the following facts.

\begin{Prop}
Let  $u,v \in S_n$ be permutations with separated descents   at position $k$.
Then, for $w\in S_n$,
\begin{itemize}
    \item[(1)] The set $L(u, v, w)$ of admissible states are in bijection with the set $\PP(u,v,w)$  of pipe puzzles.

    \item[(2)] We have 
    $$Z_{u, v}^w(t,y)=\sum_{\pi\in \PP(u,v,w)} \wt(\pi).$$
   That is, the partition function $Z_{u, v}^w(t,y)$ coincides with the right-hand side 
 of \eqref{PPKYC}. 
\end{itemize}
\end{Prop}

We next introduce  two types of $R$-matrices: $R_{\mathrm{row}}$ and  $R_{\mathrm{col}}$, and check that the lattice model satisfies the Yang--Baxter equation when attaching an  $R_{\mathrm{row}}$ (resp., an $R_{\mathrm{col}}$)
 to rows (resp., columns). 

\subsection{The $R$-matrix $R_{\mathrm{row}}$}

The  $R$-matrix  $R_{\mathrm{row}}$ is given in 
 Table \ref{tab:Rmat}. 

\begin{table}[hhhhh]
$$\begin{array}{c|c}\hline
A & B_1\\\hline
\RmatB{0}{0}{0}{0}
\RmatB{a}{a}{a}{a} & 
\RmatB{b}{0}{b}{0}
\RmatB{0}{c}{0}{c}
\RmatB{a}{A}{a}{A}
\\\hline\vphantom{\dfrac{1}{2}}
1\quad (0<a) &
1 \quad (0<b\leq k<c \text{ and }0<a<A)
\\\hline
\hline
C & B_2 \\\hline
\RmatB{0}{a}{a}{0}
\RmatB{A}{a}{a}{A}& 
\RmatB{0}{b}{0}{b}
\RmatB{c}{0}{c}{0}
\RmatB{A}{a}{A}{a}
\\\hline\vphantom{\dfrac{1}{2}}
x\quad (0<a<A)
& 1+\beta x\quad (0<b\leq k<c \text{ and }0<a<A)
\\\hline
\end{array}$$
    \caption{The $R$-matrix $R_{\mathrm{row}}$.}
    \label{tab:Rmat}
\end{table}

\begin{Th}[Yang--Baxter Equation for $R_{\mathrm{row}}$]\label{thm:YBE-1}
For the  $R$-matrix $R_{\mathrm{row}}$, the partition functions of the following two models  are equal for any given boundary condition with $a_1, a_2, a_3, b_1,$ $
b_2, b_3\in \{0, 1,2, \ldots, n\}.$ 
\begin{align}\label{eq:YBEforR}
\begin{array}{c}
\xymatrix@=0.5pc{
    &&& \C{a_1} &&\\
    \C{b_1} &&
    \C{}&x
    \ar@{-}[u]\ar@{-}[d]
    \ar@{-}[l]\ar@{-}[r]
    &\C{a_2}\\
    & \makebox[1pc][c]{\mbox{$x\ominus y$}}
    \ar@{-}@/_/[ul]\ar@{-}@/^/[ur]
    \ar@{-}@/_/[dr]\ar@{-}@/^/[dl]&&\C{} \\
    \C{b_2} & 
    &\C{}&y
    \ar@{-}[u]\ar@{-}[d]
    \ar@{-}[l]\ar@{-}[r]&
    \C{a_3}\\
    &&& \C{b_3}
    }
\end{array}
=\begin{array}{c}
    \xymatrix@=0.5pc{
    & \C{a_1} &&\\
    \C{b_1} &y
    \ar@{-}[u]\ar@{-}[d]
    \ar@{-}[l]\ar@{-}[r]&
    \C{}&&
    \C{a_2}\\
    & \C{} && 
    \makebox[1pc][c]{\mbox{$x\ominus y$}}
    \ar@{-}@/_/[ul]\ar@{-}@/^/[ur]
    \ar@{-}@/_/[dr]\ar@{-}@/^/[dl]\\
    \C{b_2} & x 
    \ar@{-}[u]\ar@{-}[d]
    \ar@{-}[l]\ar@{-}[r]&
    \C{}&&
    \C{a_3}\\
    & \C{b_3}
    }
\end{array}
\end{align}
Here, the partition function of the left (resp., right) model is the sum of all weights of admissible configurations of the left (resp., right) diagram with the given boundary condition.  
\end{Th}

\begin{proof}
Note that \eqref{eq:YBEforR} only depends on the relative values of $a_1, a_2, a_3, b_2, b_2, b_3$. By Tables \ref{tab:Lmat} and   \ref{tab:Rmat}, it suffices to  assume $a_1, a_2, a_3, b_2, b_2, b_3\in \{0, 1,2, \ldots, 6\}$ and $\#\{a_1, a_2, a_3, $ $b_2, b_2, b_3\}=3.$  So there are only finitely many cases to consider, which can be directly dealt with via computer verification.  
\end{proof} 

\begin{Eg}Take $k=3$. 
Let $(a_1,a_2,a_3,b_1,b_2,b_3)=(1,0,4,1,4,0)$. 
By Tables \ref{tab:Lmat} and \ref{tab:Rmat}, it can be checked that the admissible configurations of both sides are illustrated below.  
$$\begin{array}{c}
\xymatrix@=0.5pc{
    &&& \C{1} &\\
    \C{1} &&
    \C{1}&{x}
    \ar@{-}[u]\ar@{-}[d]
    \ar@{-}[l]\ar@{-}[r]
    &\C{0}\\
    & \makebox[1pc][c]{\mbox{$x\ominus y$}}
    \ar@{-}@/_/[ul]\ar@{-}@/^/[ur]
    \ar@{-}@/_/[dr]\ar@{-}@/^/[dl]&&\C{0} \\
    \C{4} & 
    &\C{4}&{y}
    \ar@{-}[u]\ar@{-}[d]
    \ar@{-}[l]\ar@{-}[r]&
    \C{4}\\
    &&& \C{0}
    }
\end{array}
=
\begin{array}{c}
    \xymatrix@=0.5pc{
    & \C{1} &&\\
    \C{1} &{y}
    \ar@{-}[u]\ar@{-}[d]
    \ar@{-}[l]\ar@{-}[r]&
    \C{4}&&
    \C{0}\\
    & \C{4} && 
     \makebox[1pc][c]{\mbox{$x\ominus y$}}
    \ar@{-}@/_/[ul]\ar@{-}@/^/[ur]
    \ar@{-}@/_/[dr]\ar@{-}@/^/[dl]\\
    \C{4} & {x} 
    \ar@{-}[u]\ar@{-}[d]
    \ar@{-}[l]\ar@{-}[r]&
    \C{0}&&
    \C{4}\\
    & \C{0}}
\end{array}+
\begin{array}{c}
    \xymatrix@=0.5pc{
    & \C{1} &&\\
    \C{1} &{y}
    \ar@{-}[u]\ar@{-}[d]
    \ar@{-}[l]\ar@{-}[r]&
    \C{0}&&
    \C{0}\\
    & \C{0} && 
     \makebox[1pc][c]{\mbox{$x\ominus y$}}
    \ar@{-}@/_/[ul]\ar@{-}@/^/[ur]
    \ar@{-}@/_/[dr]\ar@{-}@/^/[dl]\\
    \C{4} & {x} 
    \ar@{-}[u]\ar@{-}[d]
    \ar@{-}[l]\ar@{-}[r]&
    \C{4}&&
    \C{4}\\
    & \C{0}
    }
\end{array}
 $$
Again, in view of Tables \ref{tab:Lmat} and \ref{tab:Rmat}, the partition function of the left model is $1+\beta x $, while the partition function of the right model 
is 
$\beta (1+\beta y)(x\ominus y)
+(1+\beta y)$.  These two partition functions are indeed the same. This agrees to the assertion in \eqref{eq:YBEforR}.
\end{Eg}
  
\subsection{The $R$-matrix $R_{\mathrm{col}}$}

The $R$-matrix $R_{\mathrm{col}}$ is given  in Table \ref{tab:rmat}. 
\begin{table}[hhhhh]
$$\begin{array}{c|c}\hline
A & B_1\\\hline
\RmatC{0}{0}{0}{0}
\RmatC{a}{a}{a}{a} & 
\RmatC{b}{0}{b}{0}
\RmatC{0}{c}{0}{c}
\RmatC{a}{A}{a}{A}
\\\hline\vphantom{\dfrac{1}{2}}
1\quad (0<a) &
1 \quad (0<b\leq k<c \text{ and }0<a<A)
\\\hline
\hline
C & B_2 \\\hline
\RmatC{a}{0}{0}{a}
\RmatC{A}{a}{a}{A}& 
\RmatC{0}{b}{0}{b}
\RmatC{c}{0}{c}{0}
\RmatC{A}{a}{A}{a}
\\\hline\vphantom{\dfrac{1}{2}}
x\quad (0<a<A)
& 1+\beta x \quad (0<b\leq k<c \text{ and }0<a<A)
\\\hline
\end{array}$$
    \caption{The $R$-matrix $R_{\mathrm{col}}$.}
    \label{tab:rmat}
\end{table}

\begin{Th}[Yang--Baxter Equation for $R_{\mathrm{col}}$]\label{thm:YBE-2}
For the $R$-matrix $R_{\mathrm{col}}$, the partition functions of the following two models  are equal for any given boundary condition with $a_1, a_2,$ $ a_3, b_2, b_2, b_3\in \{0, 1,2, \ldots, n\}.$ 

\begin{equation}\label{eq:YBEforr}
\begin{array}{c}
\xymatrix@=0.5pc{
    &\C{a_1}&&\C{a_2}\\
    && \phantom{\rule{1.75pc}{1pc}}
    \ar@{}[]|{\makebox[1pc][c]{\mbox{$x\ominus y$}}}
    \ar@{-}@/_/[ur]\ar@{-}@/^/[ul]
    \ar@{-}@/_/[dl]\ar@{-}@/^/[dr]\\
    &\C{} && \C{}\\
    \C{b_1}& x
        \ar@{-}[u]\ar@{-}[d]
        \ar@{-}[l]\ar@{-}[r]
    &\C{}& y
        \ar@{-}[u]\ar@{-}[d]
        \ar@{-}[l]\ar@{-}[r]
    & \C{a_3}\\
    &\C{b_2}&&\C{b_3}
    }
\end{array}
=
\begin{array}{c}
\xymatrix@=0.5pc{
    &\C{a_1} && \C{a_2}\\
    \C{b_1}& y
        \ar@{-}[u]\ar@{-}[d]
        \ar@{-}[l]\ar@{-}[r]
    &\C{}&x
        \ar@{-}[u]\ar@{-}[d]
        \ar@{-}[l]\ar@{-}[r]
    & \C{a_3}\\
    &\C{}&&\C{}\\
    && \phantom{\rule{1.75pc}{1pc}}
    \ar@{}[]|{\makebox[1pc][c]{\mbox{$x\ominus y$}}}
    \ar@{-}@/_/[ur]\ar@{-}@/^/[ul]
    \ar@{-}@/_/[dl]\ar@{-}@/^/[dr]\\
    &\C{b_2}&&\C{b_3}
    }
\end{array}
\end{equation}
\end{Th}

\begin{Eg}
Still take  $k=3$, and $(a_1,a_2,a_3,b_1,b_2,b_3)=(1,0,4,1,4,0)$. 
Equality  \eqref{eq:YBEforr} tells
$$\beta(1+\beta x)=
\beta(1+\beta y)(1+\beta (x\ominus y)),
$$
as implied  by the following admissible  configurations
$$ 
\begin{array}{c}
\xymatrix@=0.5pc{
    &\C{1}&&\C{0}\\
    && \phantom{\rule{1.75pc}{1pc}}
    \ar@{}[]|{\makebox[1pc][c]{\mbox{$x\ominus y$}}}
    \ar@{-}@/_/[ur]\ar@{-}@/^/[ul]
    \ar@{-}@/_/[dl]\ar@{-}@/^/[dr]\\
    &\C{1} && \C{0}\\
    \C{1}& {x}
        \ar@{-}[u]\ar@{-}[d]
        \ar@{-}[l]\ar@{-}[r]
    &\C{4}& {y}
        \ar@{-}[u]\ar@{-}[d]
        \ar@{-}[l]\ar@{-}[r]
    & \C{4}\\
    &\C{4}&&\C{0}
    }
\end{array}
=
\begin{array}{c}
\xymatrix@=0.5pc{
    &\C{1} && \C{0}\\
    \C{1}& {y}
        \ar@{-}[u]\ar@{-}[d]
        \ar@{-}[l]\ar@{-}[r]
    &\C{4}& {x}
        \ar@{-}[u]\ar@{-}[d]
        \ar@{-}[l]\ar@{-}[r]
    & \C{4}\\
    &\C{4}&&\C{0}\\
    &&\phantom{\rule{1.75pc}{1pc}}
    \ar@{}[]|{\makebox[1pc][c]{\mbox{$x\ominus y$}}}
    \ar@{-}@/_/[ur]\ar@{-}@/^/[ul]
    \ar@{-}@/_/[dl]\ar@{-}@/^/[dr]\\
    &\C{4}&&\C{0}
    }
\end{array}$$
\end{Eg}

\section{Proof of the main result}\label{EFBG}

We always set $u, v, w\in S_n$ with $u, v$ owning   separated descents at position  $k$. 
We finish the proof of Theorem  \ref{thm:mainTh} by showing  that  $Z_{u, v}^w(t, y)$ satisfies the same recurrence relations 
(Propositions \ref{indonu} and  \ref{indonw}) and  initial condition (Proposition \ref{Prop:initial})  as $c_{u, v}^w(t, y)$.

\subsection{Induction on $u$}

Suppose that $s_iu<u$. It is easily checked that 
$\mathrm{maxdes}(s_iu)\leq k$. 
Recalling the definition in  \eqref{eq:boarduvw},  we see that
$0<\kappa_u^{i+1}<\kappa_u^{i}\leq k$ or 
$\kappa_u^{i}=0<\kappa_u^{i+1}\leq k$, depending on the positions where $i$ and 
$i+1$ lie. Clearly, $\kappa_{s_iu}$ is obtained from $\kappa_{u}$ by interchanging 
$\kappa_u^{i}$ and $\kappa_u^{i+1}$.
For example, for $n=7$ and  $k=3$,
we list a descending  chain as follows:
$$
\def\o#1{{\color{lightgray}#1}}
\begin{array}{c}
\ \ u=
543\o{1267\cdots}\,>\,
534\o{1267\cdots}\,>\,
524\o{1367\cdots}\,>\,
425\o{1367\cdots}\\
\kappa_u=
0032100\cdots\rightarrow 
0023100\cdots\rightarrow
\ 0203100\cdots\rightarrow
0201300\cdots
\end{array}
$$

\begin{Th}\label{thm:indonu}
If $s_iu<u$, then 
\begin{equation}\label{TREW}
    Z_{s_iu,v}^w=-\frac{1+\beta y_i}{y_i-y_{i+1}}Z_{u,v}^w
    +\frac{1+\beta y_{i+1}}{y_i-y_{i+1}}Z_{u, v}^w|_{y_i\leftrightarrow y_{i+1}}.
\end{equation} 
\end{Th}
\begin{proof} 
Consider the lattice model $L(u, v, w)$. 
We attach an $R_{\mathrm{row}}$ to the left boundary of row $i$ and row $i+1$ (Meanwhile, we make the variable exchange   $y_i\leftrightarrow y_{i+1}$ in the states of $L(u, v, w)$), as illustrated in  \eqref{eq:proofinduA}. 
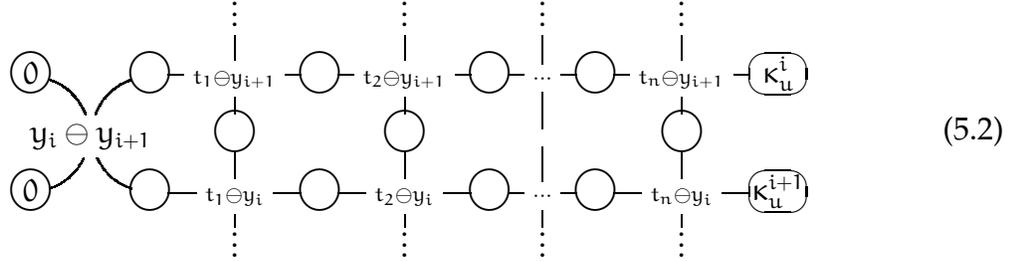
\begin{equation}\label{eq:proofinduA}
    \begin{array}{c}
    \def\objectstyle#1{\scriptstyle #1}
    \xymatrix@=0.5pc{
    &&&\vdots&&\vdots&&\vdots&&\vdots\\
    \C{0}&&\C{}&
    t_1\ominus y_{i+1}
        \ar@{-}[r]\ar@{-}[l]
        \ar@{-}[u]\ar@{-}[d]
    &\C{}&
    t_2\ominus y_{i+1}
        \ar@{-}[r]\ar@{-}[l]
        \ar@{-}[u]\ar@{-}[d]
    &\C{}&
    \cdots
        \ar@{-}[r]\ar@{-}[l]
        \ar@{-}[u]\ar@{-}[d]
    &\C{}&
    t_n\ominus y_{i+1}
        \ar@{-}[r]\ar@{-}[l]
        \ar@{-}[u]\ar@{-}[d]
    &\CC{\kappa_{u}^{i}}
    \\
    & \makebox[1pc]{\mbox{$y_i\ominus y_{i+1}$}}
    \ar@{-}@/_/[ul]\ar@{-}@/^/[ur]
        \ar@{-}@/_/[dr]\ar@{-}@/^/[dl]&&\C{}&&\C{}&&&&\C{}&\\
    \C{0}&&\C{}&
    t_1\ominus y_i
        \ar@{-}[r]\ar@{-}[l]
        \ar@{-}[u]\ar@{-}[d]
    &\C{}&
    t_2\ominus y_i
        \ar@{-}[r]\ar@{-}[l]
        \ar@{-}[u]\ar@{-}[d]
    &\C{}&
    \cdots
        \ar@{-}[r]\ar@{-}[l]
        \ar@{-}[u]\ar@{-}[d]
    &\C{}&
    t_n\ominus y_{i}
        \ar@{-}[r]\ar@{-}[l]
        \ar@{-}[u]\ar@{-}[d]
    &\CC{\kappa_{u}^{i+1}}
    \\
    &&&\vdots&&\vdots&&\vdots&&\vdots
    }
    \end{array}
\end{equation}
By Table \ref{tab:Rmat}, there is exactly  one admissible configuration for the $R$-matrix $R_{\mathrm{row}}$ (from $A$ in Table \ref{tab:Rmat}).
So the partition function of  \eqref{eq:proofinduA} reads as 
\begin{equation}\label{SL-1}
    Z_{u, v}^w|_{y_i\leftrightarrow y_{i+1}}.
\end{equation}

Noticing that $(t_j\ominus y_{i+1})\ominus (t_j\ominus y_i)=y_i\ominus y_{i+1}$, we may apply repeatedly the Yang--Baxter equation in Theorem  \ref{thm:YBE-1} to \eqref{eq:proofinduA}, resulting in a model depicted in \eqref{eq:proofinduB},  with an $R$-matrix $R_{\mathrm{row}}$ attached on the right boundary. 
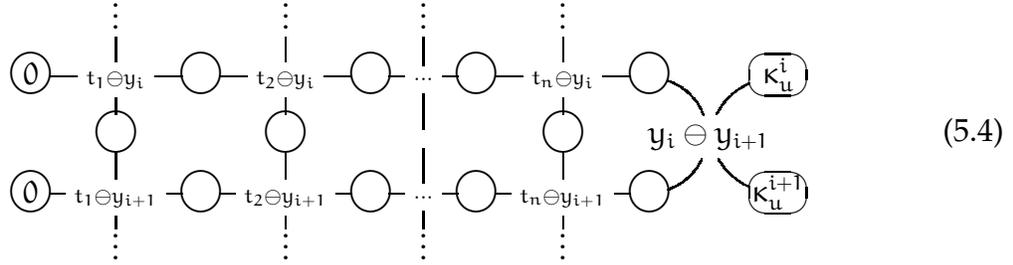
\begin{equation}\label{eq:proofinduB}
\begin{array}{c}
\def\objectstyle#1{\scriptstyle #1}
\xymatrix@=0.5pc{
&\vdots&&\vdots&&\vdots&&\vdots\\
\C{0}&
t_1\ominus y_{i}
    \ar@{-}[r]\ar@{-}[l]
    \ar@{-}[u]\ar@{-}[d]
&\C{}&
t_2\ominus y_{i}
    \ar@{-}[r]\ar@{-}[l]
    \ar@{-}[u]\ar@{-}[d]
&\C{}&
\cdots
    \ar@{-}[r]\ar@{-}[l]
    \ar@{-}[u]\ar@{-}[d]
&\C{}&
t_n\ominus y_{i}
    \ar@{-}[r]\ar@{-}[l]
    \ar@{-}[u]\ar@{-}[d]
&\C{}&&\CC{\kappa_{u}^{i}}
\\
&\C{}&&\C{}&&&&\C{}&& \makebox[1pc]{\mbox{$y_i\ominus y_{i+1}$}}
\ar@{-}@/_/[ul]\ar@{-}@/^/[ur]
    \ar@{-}@/_/[dr]\ar@{-}@/^/[dl]&\\
\C{0}&
t_1\ominus y_{i+1}
    \ar@{-}[r]\ar@{-}[l]
    \ar@{-}[u]\ar@{-}[d]
&\C{}&
t_2\ominus y_{i+1}
    \ar@{-}[r]\ar@{-}[l]
    \ar@{-}[u]\ar@{-}[d]
&\C{}&
\cdots
    \ar@{-}[r]\ar@{-}[l]
    \ar@{-}[u]\ar@{-}[d]
&\C{}&
t_n\ominus y_{i+1}
    \ar@{-}[r]\ar@{-}[l]
    \ar@{-}[u]\ar@{-}[d]
&\C{}&&\CC{\kappa_{u}^{i+1}}
\\
&\vdots&&\vdots&&\vdots&&\vdots
}
\end{array}
\end{equation}

Consider the partition function of \eqref{eq:proofinduB}.
Keep in mind that 
$0<\kappa_u^{i+1}<\kappa_u^{i}\leq k$ or 
$\kappa_u^{i}=0<\kappa_u^{i+1}\leq k$.
For each situation, there are two admissible configurations for the $R$-matrix $R_{\mathrm{row}}$ respectively from 
$B_2$ and $C$ in Table \ref{tab:Rmat},  corresponding  respectively to the models  $L(u, v, w)$ and  $L(s_iu, v, w)$.
Thus, the partition function of    \eqref{eq:proofinduB} is
\begin{equation}\label{SL-2}
   (1+\beta(y_i\ominus y_{i+1}))
Z_{u,v}^w
+(y_i\ominus y_{i+1})
Z_{s_iu,v}^w. 
\end{equation}
Equating \eqref{SL-1} and \eqref{SL-2}, 
we get the desired formula in \eqref{TREW}.
\end{proof}

\subsection{Induction on $w$}

We now establish the recurrence relation for $Z_{u, v}^w$, which is parallel to Proposition    \ref{indonw}.

\begin{Th}\label{thm:indonw}
If $s_iw>w$, then 
\begin{align}\label{REWQ}
Z_{u,v}^{s_iw} & 
=\begin{cases}
  -\dfrac{1+\beta t_{i+1}}{t_i-t_{i+1}}Z_{u,v}^w|_{t_i\leftrightarrow t_{i+1}}
+\dfrac{1+\beta t_{i}}{t_i-t_{i+1}}Z_{u,v}^w
+ Z_{u,s_iv}^w|_{t_i\leftrightarrow t_{i+1}}, 
& s_iv<v,\\[15pt]
  -\dfrac{1+\beta t_{i}}{t_i-t_{i+1}}Z_{u, v}^w|_{t_i\leftrightarrow t_{i+1}}
+\dfrac{1+\beta t_{i}}{t_i-t_{i+1}}Z_{u, v}^w,
& s_iv>v. 
\end{cases}
\end{align}
\end{Th}

\begin{proof}

This time we attach an  $R_{\mathrm{col}}$ to the top boundary of $L(u, v, w)$.
Applying the Yang--Baxter equation in Theorem \ref{thm:YBE-2}, we obtain equivalent models given in \eqref{eq:proofinduC}.
\begin{equation}\label{eq:proofinduC}
    \begin{array}{c}
    \def\objectstyle#1{\scriptstyle #1}
    \xymatrix@=0.5pc{
    &\CC{\theta_v^i}&&\CC{\theta_v^{i+1}}\\
    &&\phantom{\rule{0pc}{1pc}}
        \ar@{-}@/_/[ur]\ar@{-}@/^/[ul]
        \ar@{-}@/_/[dl]\ar@{-}@/^/[dr]
        \ar@{}[]|{\makebox[0pc]{\colorbox{white}{$t_i\ominus t_{i+1}$}}}\\
    &\C{}&&\C{}\\
    \cdots&
    t_{i}\ominus y_1
        \ar@{-}[r]\ar@{-}[l]
        \ar@{-}[u]\ar@{-}[d]
    &\C{}&
    t_{i+1}\ominus y_1
        \ar@{-}[r]\ar@{-}[l]
        \ar@{-}[u]\ar@{-}[d]
    &\cdots\\
    &\C{}&&\C{}\\
    \cdots&
    \vdots
        \ar@{-}[r]\ar@{-}[l]
        \ar@{-}[u]\ar@{-}[d]
    &&
    \vdots
        \ar@{-}[r]\ar@{-}[l]
        \ar@{-}[u]\ar@{-}[d]
    &\cdots\\
    &\C{}&&\C{}\\
    \cdots&
    t_{i}\ominus y_n
        \ar@{-}[r]\ar@{-}[l]
        \ar@{-}[u]\ar@{-}[d]
    &\C{}&
    t_{i+1}\ominus y_n
        \ar@{-}[r]\ar@{-}[l]
        \ar@{-}[u]\ar@{-}[d]
    &\cdots\\
    &\CC{\eta_w^{i}}&&\CC{\eta_w^{i+1}}\\
    }
    \end{array}
    = 
        \begin{array}{c}
    \def\objectstyle#1{\scriptstyle #1}
    \xymatrix@=0.5pc{
    &\CC{\theta_v^i}&&\CC{\theta_v^{i+1}}\\
    \cdots&
    t_{i+1}\ominus y_1
        \ar@{-}[r]\ar@{-}[l]
        \ar@{-}[u]\ar@{-}[d]
    &\C{}&
    t_{i}\ominus y_1
        \ar@{-}[r]\ar@{-}[l]
        \ar@{-}[u]\ar@{-}[d]
    &\cdots\\
    &\C{}&&\C{}\\
    \cdots&
    \vdots
        \ar@{-}[r]\ar@{-}[l]
        \ar@{-}[u]\ar@{-}[d]
    &&
    \vdots
        \ar@{-}[r]\ar@{-}[l]
        \ar@{-}[u]\ar@{-}[d]
    &\cdots\\
    &\C{}&&\C{}\\
    \cdots&
    t_{i+1}\ominus y_n
        \ar@{-}[r]\ar@{-}[l]
        \ar@{-}[u]\ar@{-}[d]
    &\C{}&
    t_{i}\ominus y_n
        \ar@{-}[r]\ar@{-}[l]
        \ar@{-}[u]\ar@{-}[d]
    &\cdots\\
    &\C{}&&\C{}\\
    &&\phantom{\rule{0pc}{1pc}}
        \ar@{-}@/_/[ur]\ar@{-}@/^/[ul]
        \ar@{-}@/_/[dl]\ar@{-}@/^/[dr]
        \ar@{}[]|{\makebox[0pc]{\colorbox{white}{$t_i\ominus t_{i+1}$}}}\\
    &\CC{\eta_w^{i}}&&\CC{\eta_w^{i+1}}\\
    }
    \end{array}
\end{equation}

We first consider the partition function of the right model in \eqref{eq:proofinduC}. The assumption  $s_iw>w$ implies    
$0<\eta_w^i<\eta_{w}^{i+1}$. 
Notice also that 
$\eta_{s_iw}$ is obtained from $\eta_w$ by interchanging $\eta_{w}^i$ and $\eta_w^{i+1}$. 
In view of  Table \ref{tab:rmat}, there  are two admissible configurations  for the  $R$-matrix $R_{\mathrm{col}}$ (one   is from  $B_1$ in Table \ref{tab:rmat}, and the other is from $C$ in Table \ref{tab:rmat}), corresponding respectively to the models 
$L(u, v, w)$ and $L(u, v, s_iw)$.
So the partition function of the right model in \eqref{eq:proofinduC} is 
\begin{equation}\label{PPYYTT}
 Z_{u, v}^{w}|_{t_i\leftrightarrow t_{i+1}}
+(t_i\ominus t_{i+1})Z_{u, v}^{s_iw}|_{t_i\leftrightarrow t_{i+1}}.   
\end{equation}

We next consider the   partition function of the left model in \eqref{eq:proofinduC}.
There are two cases.

Case 1. 
$s_iv<v$. In this case, notice that 
$k<\theta_u^{i+1}<\theta_u^{i}$ or 
$0=\theta_u^{i+1}<\theta_u^{i}$, and that
$\theta_{s_iv}$ is obtained from $\theta_{v}$ by interchanging 
$\theta_v^{i}$ and $\theta_v^{i+1}$. 
By Table \ref{tab:rmat}, for either $k<\theta_u^{i+1}<\theta_u^{i}$ or 
$0=\theta_u^{i+1}<\theta_u^{i}$, there are two choices for the configurations of $R_{\mathrm{col}}$ (one   is from $B_2$, and the other is from $C$), corresponding respectively to the models 
$L(u, v, w)$ and $L(u, s_iv, w)$.
So, the partition function of the left model in \eqref{eq:proofinduC} is 
\begin{equation}\label{PPYYTT-1}
    (1+\beta (t_i\ominus t_{i+1}))Z_{u, v}^{w}
+(t_i\ominus t_{i+1})Z_{u,s_iv}^{w}.
\end{equation}
Equating \eqref{PPYYTT} and \eqref{PPYYTT-1}, we deduce that 
\begin{align*}
Z_{u, v}^{s_iw}|_{t_i\leftrightarrow t_{i+1}}
& = 
\frac{1+\beta (t_i\ominus t_{i+1})}{t_i\ominus 
t_{i+1}}Z_{u, v}^{w}+Z_{u,s_iv}^{w}
-\frac{1}{t_i\ominus t_{i+1}}Z_{u, v}^{w}|_{t_i\leftrightarrow t_{i+1}}
\\[5pt]
& = 
\frac{1+\beta t_i}{t_i-t_{i+1}}Z_{u, v}^{w}
+Z_{u,s_iv}^{w}-\frac{1+\beta t_{i+1}}{t_i-t_{i+1}}Z_{u, v}^{w}|_{t_i\leftrightarrow t_{i+1}},
\end{align*}
which, after the variable exchange $t_i\leftrightarrow t_{i+1}$, becomes the first equality in \eqref{REWQ}.

Case 2.    $s_iv>v$. 
In this case, $i$ appears before $i+1$ in $v$. So we have $0=\theta_v^i=\theta_v^{i+1}$, or $0=\theta_v^i$ and $k<\theta_v^{i+1}$, or 
$k<\theta_v^i<\theta_v^{i+1}$.
By Table \ref{tab:rmat}, for each of these   situations,  there is exactly one admissible configuration (from $A$ or $B_1$) of $R_{\mathrm{col}}$, and  we see that the partition function of the left model in \eqref{eq:proofinduC} is precisely  equal to $  Z_{u,v}^{w}$. By equating with \eqref{PPYYTT}, we obtain that  
\begin{align*}
Z_{u, v}^{s_iw}|_{t_i\leftrightarrow t_{i+1}}
& = 
\frac{1}{t_i\ominus t_{i+1}}Z_{u, v}^{w}
-\frac{1}{t_i\ominus t_{i+1}}Z_{u, v}^{w}|_{t_i\leftrightarrow t_{i+1}}\\[5pt]
& = 
\frac{1+\beta t_{i+1}}{t_i-t_{i+1}}Z_{u, v}^{w}
-\frac{1+\beta t_{i+1}}{t_i-t_{i+1}}Z_{u, v}^{w}|_{t_i\leftrightarrow t_{i+1}}.
\end{align*}
After the variable exchange  $t_i\leftrightarrow t_{i+1}$ on both sides, we reach the second equality in \eqref{REWQ}.
\end{proof}

\subsection{Initial condition}

We finally verify the initial case for $u_0$ (as  defined in \eqref{eq:defu0}) and $w=\operatorname{id}$.

\begin{Th}\label{thm:initial}
For   $v\in S_n$, we have 
\begin{equation}\label{TGB}
    Z_{u_0, v}^{\operatorname{id}}
    = \begin{cases}\displaystyle
        \prod_{i=1}^k\prod_{j=1}^{n-i}(t_i\ominus y_j), & \text{if }v=\operatorname{id},\\
        0, & \text{otherwise}.
    \end{cases}
\end{equation}
\end{Th}

\begin{proof}
Here, we go back to the pipe puzzle model $\PP(u_0, v, \operatorname{id})$ for the computation of $Z_{u_0, v}^{\operatorname{id}}$.
The boundary condition is illustrated in the left diagram in  Figure  \ref{fig:enter-label-AB}.
\begin{figure}[ht]
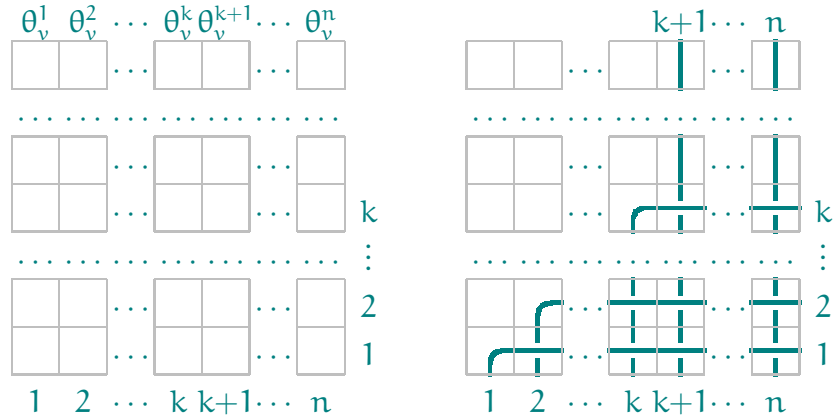

    \centering
$$\BPD[1.5pc]{
\M{\theta_v^1}\M{\theta_v^2}\M{\cdots}\M{\theta_v^k}\M{\theta_v^{k{+}1}}\M{\cdots}\M{\theta_v^n}\\
\+\+\M{\cdots}\+\+\M{\cdots}\+\M{}\\
\M{\cdots}\M{\cdots}\M{\cdots}
\M{\cdots}\M{\cdots}\M{\cdots}
\M{\cdots}\M{}\\
\+\+\M{\cdots}\+\+\M{\cdots}\+\M{}\\
\+\+\M{\cdots}\+\+\M{\cdots}\+\M{k}\\
\M{\cdots}\M{\cdots}\M{\cdots}\M{\cdots}\M{\cdots}\M{\cdots}\M{\cdots}\M{\vdots}\\
\+\+\M{\cdots}\+\+\M{\cdots}\+\M{2}\\
\+\+\M{\cdots}\+\+\M{\cdots}\+\M{1}\\
\M{1}\M{2}\M{\cdots}\M{k}\M{k{+}1}\M{\cdots}\M{n}
}\qquad
\BPD[1.5pc]{
\M{}\M{}\M{}\M{}\M{k{+}1}\M{\cdots}\M{n}\\
\O\O\M{\cdots}\O\I\M{\cdots}\I\\
\M{\cdots}\M{\cdots}\M{\cdots}\M{\cdots}
\M{\cdots}\M{\cdots}\M{\cdots}\\
\O\O\M{\cdots}\O\I\M{\cdots}\I\\
\O\O\M{\cdots}\F\X\M{\cdots}\X\M{k}\\
\M{\cdots}\M{\cdots}\M{\cdots}\M{\cdots}\M{\cdots}\M{\cdots}\M{\cdots}\M{\vdots}\\
\O\F\M{\cdots}\X\X\M{\cdots}\X\M{2}\\
\F\X\M{\cdots}\X\X\M{\cdots}\X\M{1}\\
\M{1}\M{2}\M{\cdots}\M{k}\M{k{+}1}\M{\cdots}\M{n}
}
$$   
    \caption{Boundary condition and the unique pipe puzzle in $\PP(u_0,\operatorname{id},\operatorname{id})$.}
    \label{fig:enter-label-AB}
\end{figure}
Evidently, the pipes labeled  $k+1,\ldots,n$ must go vertically from the top side down  to the bottom side.  So we have $Z_{u_0, v}^{\operatorname{id}}=0$ whenever $v\neq \operatorname{id}$. It remains to check the case    $v= \operatorname{id}$.
It is easily checked  that
there is exactly one pipe puzzle in $\PP(u_0,\operatorname{id},\operatorname{id})$, see the  right diagram of Figure \ref{fig:enter-label-AB}. This pipe puzzle contributes a weight as displayed in \eqref{TGB}.
\end{proof}

\section{Applications}\label{TGBF}

We list three main applications of  Theorem \ref{thm:mainTh}. The first application  is to recover the puzzle formula discovered   by Knutson and Zinn-Justin \cite[Theorem 1]{Pzz3}. 

\subsection{Separated-descent puzzles}

Consider  \eqref{eq:cuvwtoind} by setting $y=t$:
$$\G_{u}(x,t)\cdot \G_{v}(x,t)=\sum_w c_{u, v}^w(t,t)\cdot \G_{w}(x,t).$$
Assume that  $u, v\in S_n$ have separated descents at $k$. 
For $w\in S_n$, a pipe puzzle $\pi\in \PP(u,v,w)$ has weight zero if and only if $\pi$ has (at least) one empty  tile  $\BPD{\O}$ on the diagonal. This implies that $c_{u, v}^w(t,t)$   is a weighted counting of pipe puzzles $\pi\in \PP(u,v,w)$ such that $\pi$ has no empty tile on the diagonal. 
For such pipe puzzles, we have the following observation:
\begin{itemize}
    \item Each position  on the diagonal is tiled with  either $\BPD{\F}$ or $\BPD{\I}$, and each  position lying strictly to the southwest of the diagonal is tiled  with $\BPD{\I}$. 
\end{itemize}

This can be checked as follows. 
First,  the tile at the  position $(1,1)$   must be tiled with either  $\BPD{\F}$ or $\BPD{\I}$ since (1) the tile cannot be empty,  and (2) the labels on the left boundary are all $0$. 
    Therefore, all  positions below $(1,1)$ in the first column must be tiled with $\BPD{\I}$. The same analysis applies to the remaining positions $(2,2), \ldots, (n,n)$. 
   
Let  $\pi\in \PP(u,v,w)$ be a pipe puzzle without empty tile on the diagonal. 
Cut $\pi$  along its diagonal into two triangles, and denote by $P(\pi)$ the upper-right triangle.  
By the above observation, $\pi$ can be recovered from  $P(\pi)$. 
 To get the puzzle visualization of Knutson and Zinn-Justin \cite[Theorem 1]{Pzz3}, we rotate $P(\pi)$ counterclockwise by $45$ degrees, and then warp it into  an equilateral triangle.
{  If  further assuming that $u$ and $v$ are both $k$-Grassmannian, there is a direct bijection to the classical Grassmannian puzzles, see \cite[\textsection 5.1]{Pzz3} for more details.}
 
\newcommand{\R}[1][]{
\begin{picture}(1,1)%
    \linethickness{0.08\unitlength}
    \qbezier(0.5,0.5)(0.5,0.5)(1,0.5)
    \put(-0.3,-0.1){\makebox[\unitlength]{\(#1\)}}
    \thinlines
    \color{lightgray}%
    \put(1,0){\line(0,1){1}}%
    \put(0,1){\line(1,0){1}}%
    \put(1,0){\line(-1,1){1}}%
\end{picture}}
\newcommand{\U}[1][]{
\begin{picture}(1,1)%
    \linethickness{0.08\unitlength}
    \qbezier(0.5,0.5)(0.5,0.5)(0.5,1)
    \put(-0.3,-0.1){\makebox[\unitlength]{\(#1\)}}
    \thinlines
    \color{lightgray}%
    \put(1,0){\line(0,1){1}}%
    \put(0,1){\line(1,0){1}}%
    \put(1,0){\line(-1,1){1}}%
\end{picture}}


\begin{Eg}
Consider the pipe puzzles  in Example \ref{eg:pipepuzzle}. 
The following four puzzles survive after setting $y=t$.
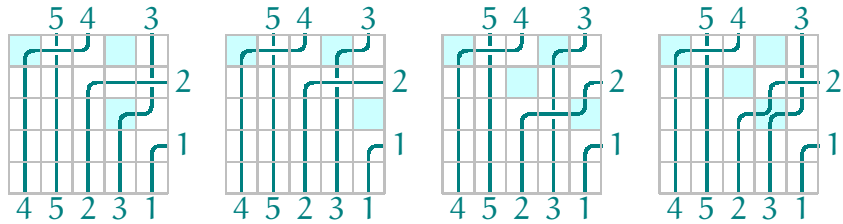
$$\def\o{\BPDfr{\put(0,0){\color{lightcyan}\rule{\unitlength}{\unitlength}}}}
\def\f{\BPDfr{\put(0,0){\color{lightcyan}\rule{\unitlength}{\unitlength}}\FF}}
\def\j{\BPDfr{\put(0,0){\color{lightcyan}\rule{\unitlength}{\unitlength}}\JJ}}
\def\b{\BPDfr{\put(0,0){\color{lightcyan}\rule{\unitlength}{\unitlength}}\FF\JJ}}
\def\Bb{\BPDfr{\put(0,0){\color{lightcyan}\rule{\unitlength}{\unitlength}}\FF\JJ}}
\BPD{\M{}\M{5}\M{4}\M{}\M{3}\\
\f\X\J\o\I\M{}\\
\I\I\F\H\X\M{2}\\
\I\I\I\f\J\M{}\\
\I\I\I\I\F\M{1}\\
\I\I\I\I\I\M{}\\
\M{4}\M{5}\M{2}\M{3}\M{1}}\ \ 
\BPD{\M{}\M{5}\M{4}\M{}\M{3}\\
\f\X\J\f\J\M{}\\
\I\I\F\X\H\M{2}\\
\I\I\I\I\o\M{}\\
\I\I\I\I\F\M{1}\\
\I\I\I\I\I\M{}\\
\M{4}\M{5}\M{2}\M{3}\M{1}}\ \ 
\BPD{\M{}\M{5}\M{4}\M{}\M{3}\\
\f\X\J\f\J\M{}\\
\I\I\o\I\F\M{2}\\
\I\I\F\X\j\M{}\\
\I\I\I\I\F\M{1}\\
\I\I\I\I\I\M{}\\
\M{4}\M{5}\M{2}\M{3}\M{1}}\ \ 
\BPD{\M{}\M{5}\M{4}\M{}\M{3}\\
\f\X\J\o\I\M{}\\
\I\I\o\F\X\M{2}\\
\I\I\F\b\J\M{}\\
\I\I\I\I\F\M{1}\\
\I\I\I\I\I\M{}\\
\M{4}\M{5}\M{2}\M{3}\M{1}}
$$
Their  upper-right triangular regions  are
$$\def\o{\BPDfr{\put(0,0){\color{lightcyan}\rule{\unitlength}{\unitlength}}}}
\def\f{\BPDfr{\put(0,0){\color{lightcyan}\rule{\unitlength}{\unitlength}}\FF}}
\def\j{\BPDfr{\put(0,0){\color{lightcyan}\rule{\unitlength}{\unitlength}}\JJ}}
\def\b{\BPDfr{\put(0,0){\color{lightcyan}\rule{\unitlength}{\unitlength}}\FF\JJ}}
\def\Bb{\BPDfr{\put(0,0){\color{lightcyan}\rule{\unitlength}{\unitlength}}\FF\JJ}}
\BPD{\M{}\M{5}\M{4}\M{}\M{3}\\
\R[4]\X\J\o\I\M{}\\
\M{}\U[5]\F\H\X\M{2}\\
\M{}\M{}\U[2]\f\J\M{}\\
\M{}\M{}\M{}\U[3]\F\M{1}\\
\M{}\M{}\M{}\M{}\U[1]\M{}}\ \ 
\BPD{\M{}\M{5}\M{4}\M{}\M{3}\\
\R[4]\X\J\f\J\M{}\\
\M{}\U[5]\F\X\H\M{2}\\
\M{}\M{}\U[2]\I\o\M{}\\
\M{}\M{}\M{}\U[3]\F\M{1}\\
\M{}\M{}\M{}\M{}\U[1]\M{}}\ \ 
\BPD{\M{}\M{5}\M{4}\M{}\M{3}\\
\R[4]\X\J\f\J\M{}\\
\M{}\U[5]\o\I\F\M{2}\\
\M{}\M{}\R[2]\X\j\M{}\\
\M{}\M{}\M{}\U[3]\F\M{1}\\
\M{}\M{}\M{}\M{}\U[1]\M{}}\ \ 
\BPD{\M{}\M{5}\M{4}\M{}\M{3}\\
\R[4]\X\J\o\I\M{}\\
\M{}\U[5]\o\F\X\M{2}\\
\M{}\M{}\R[2]\b\J\M{}\\
\M{}\M{}\M{}\U[3]\F\M{1}\\
\M{}\M{}\M{}\M{}\U[1]\M{}}
$$
After rotation and warping,   the corresponding  puzzles are
$$
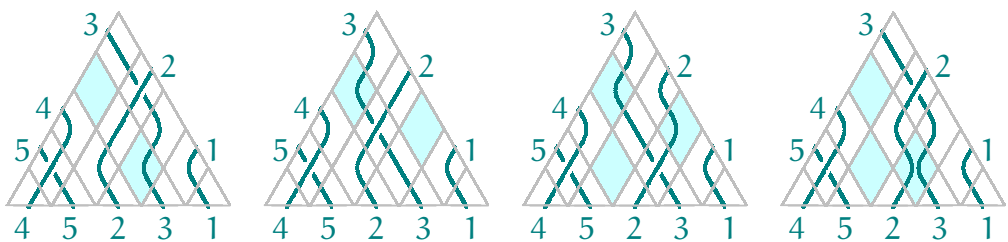
\def\rBPD#1{{\color{teal}\begin{array}{c}
{\scalebox{1}[1.73]{\rotatebox[origin=c]{45}{\(\BPD{#1}\)}}}
\\[-12.5pc]
3\quad\, \phantom{0}\\[0.15pc]
\phantom{0}\qquad\,\,\,\, 2\\[0.15pc]
4\qquad\qquad\, \phantom{0}\\[0.15pc]
5\quad\qquad\qquad\,\,\, 1\\[1.3pc]
4\quad 5\quad 2\quad 3\quad 1
\end{array}}}
\def\o{\BPDfr{\put(0,0){\color{lightcyan}\rule{\unitlength}{\unitlength}}}}
\def\f{\BPDfr{\put(0,0){\color{lightcyan}\rule{\unitlength}{\unitlength}}\FF}}
\def\j{\BPDfr{\put(0,0){\color{lightcyan}\rule{\unitlength}{\unitlength}}\JJ}}
\def\b{\BPDfr{\put(0,0){\color{lightcyan}\rule{\unitlength}{\unitlength}}\FF\JJ}}
\def\Bb{\BPDfr{\put(0,0){\color{lightcyan}\rule{\unitlength}{\unitlength}}\FF\JJ}}
\rBPD{
\R\X\J\o\I\\
\M{}\U\F\H\X\\
\M{}\M{}\U\f\J\\
\M{}\M{}\M{}\U\F\\
\M{}\M{}\M{}\M{}\U}
\rBPD{
\R\X\J\f\J\\
\M{}\U\F\X\H\\
\M{}\M{}\U\I\o\\
\M{}\M{}\M{}\U\F\\
\M{}\M{}\M{}\M{}\U}
\rBPD{
\R\X\J\f\J\\
\M{}\U\o\I\F\\
\M{}\M{}\R\X\j\\
\M{}\M{}\M{}\U\F\\
\M{}\M{}\M{}\M{}\U}
\rBPD{
\R\X\J\o\I\\
\M{}\U\o\F\X\\
\M{}\M{}\R\b\J\\
\M{}\M{}\M{}\U\F\\
\M{}\M{}\M{}\M{}\U}
$$

\end{Eg}

In the second application, we explain that  Theorem \ref{thm:mainTh} could be used to recover the bumpless pipe dream model of double Grothendieck polynomials by Weigandt \cite{Weigandt21}.

\subsection{Bumpless pipe dreams}

Let $k=n$ and $v=\operatorname{id}$. 
In this case, arbitrary $u\in S_n$ satisfies  the separated-descent  condition in \eqref{eq:sepdescent}. 
By Lemma \ref{Prop:cuvid=G}, 
$$c_{u,\operatorname{id}}^{\operatorname{id}}(t,y)=\G_u(t,y).$$
Let $\pi\in \PP(u,\operatorname{id},\operatorname{id})$. Then all pipes enter into $\pi$ from the right side. Apply the following operations to $\pi$:
\begin{itemize}

  \item reflecting $\pi$ across  the diagonal;
  
    \item replacing    $\kappa_u^i=u^{-1}(i)$  by $i$, and    $\eta_w^i=i$  by $u(i)$.

\end{itemize}
The resulting diagram is denoted as $B(\pi)$. 
Write
\[
\mathrm{BP}(u)=\{B(\pi)\colon \pi\in  \PP(u,\operatorname{id},\operatorname{id})\}.
\]
By the   restriction \eqref{eq:crossRes} on  $\BPD{\X}$  along with  the restriction \eqref{eq:bumpResA} on   $\BPD{\B}$, it can be checked that  for a diagram in  $\mathrm{BP}(u)$: (1) two pipes cross at most once, and (2) if two pipes have a  ``bumping''   $\BPD{\B}$ at position $(i,j)$, then they must cross at a position  to the northeast of $(i,j)$. 
This implies  that  the set $\mathrm{BP}(u)$
is precisely the set of bumpless pipe 
dreams of $u$, as  defined in 
\cite{Weigandt21}.

\begin{Rmk} 
As  bumpless pipe dreams in $\mathrm{BP}(u)$ are obtained from pipe puzzles  in $\PP(u,\operatorname{id},\operatorname{id})$ after a reflection, a tile at position $(i, j)$  is assigned a wight in the following way: 
\begin{enumerate}
    \item an empty tile $\BPD{\O}$  contributes $t_i\ominus y_j$; 
    
    \item  an elbow tile $\BPD{\J}$   contributes $1+\beta(t_i\ominus y_j)$;

    \item a bumping  tile  $\BPD{\B}$  contributes $\beta$;  

    \item any other  tile    contributes 1.
\end{enumerate} 
The weights     described above are slightly different from  the weights adopted  in \cite{Weigandt21}. It seems that when setting $\beta=0$,  the weights we   used imply more explicitly  the bumpless pipe dream model of double Schubert polynomials due to Lam, Lee and Shimozono \cite{LLS21}. 
\end{Rmk}

\begin{Eg}
Let $u=32514$. Below are pipe puzzles in  $\PP(u,\operatorname{id},\operatorname{id})$. 
$$
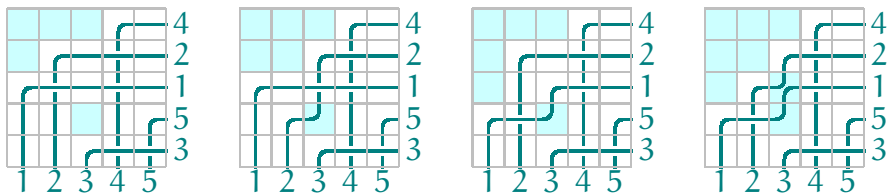
\def\o{\BPDfr{\put(0,0){\color{lightcyan}\rule{\unitlength}{\unitlength}}}}
\def\f{\BPDfr{\put(0,0){\color{lightcyan}\rule{\unitlength}{\unitlength}}\FF}}
\def\j{\BPDfr{\put(0,0){\color{lightcyan}\rule{\unitlength}{\unitlength}}\JJ}}
\def\b{\BPDfr{\put(0,0){\color{lightcyan}\rule{\unitlength}{\unitlength}}\FF\JJ}}
\def\Bb{\BPDfr{\put(0,0){\color{lightcyan}\rule{\unitlength}{\unitlength}}\FF\JJ}}
\BPD{
\o\o\o\F\H\M{4}\\
\o\F\H\X\H\M{2}\\
\F\X\H\X\H\M{1}\\
\I\I\o\I\F\M{5}\\
\I\I\F\X\X\M{3}\\
\M{1}\M{2}\M{3}\M{4}\M{5}}\ \ \ \ 
\BPD{
\o\o\o\F\H\M{4}\\
\o\o\F\X\H\M{2}\\
\F\H\X\X\H\M{1}\\
\I\F\j\I\F\M{5}\\
\I\I\F\X\X\M{3}\\
\M{1}\M{2}\M{3}\M{4}\M{5}}\ \ \ \ 
\BPD{
\o\o\o\F\H\M{4}\\
\o\F\H\X\H\M{2}\\
\o\I\F\X\H\M{1}\\
\F\X\j\I\F\M{5}\\
\I\I\F\X\X\M{3}\\
\M{1}\M{2}\M{3}\M{4}\M{5}}\ \ \ \ 
\BPD{
\o\o\o\F\H\M{4}\\
\o\o\F\X\H\M{2}\\
\o\F\Bb\X\H\M{1}\\
\F\X\j\I\F\M{5}\\
\I\I\F\X\X\M{3}\\
\M{1}\M{2}\M{3}\M{4}\M{5}}
$$
After reflection  and relabeling, the resulting bumpless pipe dreams of $u$ are
$$
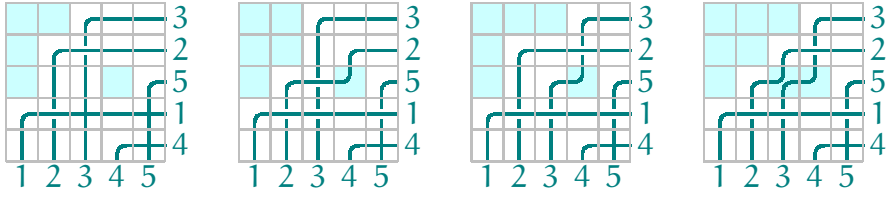
\def\o{\BPDfr{\put(0,0){\color{lightcyan}\rule{\unitlength}{\unitlength}}}}
\def\f{\BPDfr{\put(0,0){\color{lightcyan}\rule{\unitlength}{\unitlength}}\FF}}
\def\j{\BPDfr{\put(0,0){\color{lightcyan}\rule{\unitlength}{\unitlength}}\JJ}}
\def\b{\BPDfr{\put(0,0){\color{lightcyan}\rule{\unitlength}{\unitlength}}\FF\JJ}}
\def\Bb{\BPDfr{\put(0,0){\color{lightcyan}\rule{\unitlength}{\unitlength}}\FF\JJ}}
\BPD{
\o\o\F\H\H\M{3}\\
\o\F\X\H\H\M{2}\\
\o\I\I\o\F\M{5}\\
\F\X\BX\H\BX\M{1}\\
\I\I\I\F\BX\M{4}\\
\M{1}\M{2}\M{3}\M{4}\M{5}}\ \ \ \ 
\BPD{
\o\o\F\H\H\M{3}\\
\o\o\I\F\H\M{2}\\
\o\F\X\j\F\M{5}\\
\F\X\X\H\X\M{1}\\
\I\I\I\F\BX\M{4}\\
\M{1}\M{2}\M{3}\M{4}\M{5}}\ \ \ \ 
\BPD{
\o\o\o\F\H\M{3}\\
\o\F\H\X\H\M{2}\\
\o\I\F\j\F\M{5}\\
\F\X\X\H\X\M{1}\\
\I\I\I\F\BX\M{4}\\
\M{1}\M{2}\M{3}\M{4}\M{5}}\ \ \ \ 
\BPD{
\o\o\o\F\H\M{3}\\
\o\o\F\X\H\M{2}\\
\o\F\b\j\F\M{5}\\
\F\X\X\H\X\M{1}\\
\I\I\I\F\BX\M{4}\\
\M{1}\M{2}\M{3}\M{4}\M{5}}
$$
\end{Eg}

We finally apply  Theorem \ref{HHUU}
to investigate a conjecture posed by  Kirillov \cite{Kirillov07}. 

\subsection{Kirillov's conjecture}\label{KiriCon}

Let us restrict to  Schubert polynomials. 
Setting  $\beta=0$, the operator $\pi_i$ is usually denoted as $\partial_i$:
$$\partial_i f = \frac{f-f|_{x_i\leftrightarrow x_{i+1}}}{x_i-x_{i+1}}.$$
The operator $\partial_i$ is also called the \emph{divided difference operator}. 
For $w\in S_\infty$, define $\partial_w=\partial_{i_1}\cdots \partial_{i_\ell}$ for  any reduced decomposition $w=s_{i_1}\cdots s_{i_{\ell}}$(this is well defined since the $\partial_i$'s satisfy the braid relations). 
It can be deduced  that \cite[Proposition 2]{Kirillov07}
\begin{equation}\label{OKNT}
    \partial_w \S_u(x, t)= \begin{cases}\displaystyle
        \S_{u{w^{-1}}}(x, t), &  \text{if }\ell(uw^{-1})=\ell(u)-\ell(w),\\[5pt]
        0, & \text{otherwise}.
    \end{cases}
\end{equation}
The \emph{skew operator} $\partial_{w/v}$  is characterized by  
\begin{equation}\label{eq:defofskewop}
\partial_w (fg) = \sum_v (\partial_{w/v}f)(\partial_v g).
\end{equation}
See  \cite[Definition 4]{Kirillov07} for a more concrete description  of  $\partial_{w/v}$.
Kirillov \cite[Conjecture 1]{Kirillov07} conjectured that for any $u, v, w$, the polynomial $\partial_{w/v}\S_u(x)$ has nonnegative integer coefficients:
$$\partial_{w/v}\S_u(x) \in \mathbb{Z}_{\geq 0}[x_1,x_2,\ldots].$$

Setting $y=0$ in \eqref{QQSSCC} yields that 
\begin{equation}\label{RFV}
  \S_{u}(x)\cdot \S_{v}(x,t)=\sum_{w} \bar{c}_{u,v}^w(t,0)\cdot \S_w(x,t).  
\end{equation}

\begin{Prop}\label{Coro:SxxSxt}
We have 
\begin{equation*} 
 \partial_{w/v}\S_{u}(x)=\bar{c}_{u, v}^w(x,0).   
\end{equation*}
\end{Prop}

\begin{proof}
Apply  $\partial_w$ to both sides of \eqref{RFV},  and then take the specialization  $x=t$. In view of \eqref{OKNT},  \eqref{eq:defofskewop} and the localization formula in \eqref{eq:locofG} (which is still valid for double Schubert polynomials),  the left-hand side becomes $ \partial_{w/v}\S_{u}(x)$, and the right-hand side is left with $\bar{c}_{u, v}^w(x,0)$.
\end{proof}

Setting $y=0$ in Theorem \ref{HHUU}, we arrive at the following conclusion.

\begin{Coro}\label{Coro:SxxSxt-2}
Let  $u,v \in S_n$ be permutations with separated descents. Then
$$\S_u(x)\cdot\S_v(x,t)\in \sum_w \mathbb{Z}_{\geq 0}[t]\cdot\S_w(x,t).$$
\end{Coro} 

Combining Proposition \ref{Coro:SxxSxt} with 
Corollary \ref{Coro:SxxSxt-2} enables  us to confirm Kirillov's conjecture for permutations with separated descents. 

\begin{Coro}
Kirillov's conjecture is true  for $u$ and $v$ with separated descents and arbitrary $w$. 
\end{Coro}


\appendix

\section{Left Demazure operators}
\label{sec:indform}

Define \emph{the (left) Demazure operator} by 
$$\varpi_i f = -\frac{(1+\beta t_{i})f-(1+\beta t_{i+1})f|_{t_i\leftrightarrow t_{i+1}}}{t_i-t_{i+1}}.$$

\begin{Prop}\label{Prop:LDO}
We have 
\begin{align}\label{eq:varpiGw}
    \varpi_i\G_w(x,t)=\begin{cases}
    \G_{s_iw}(x,t), & \text{if }s_iw<w,\\[5pt]
    -\beta\G_{w}(x,t), & \text{if }s_iw>w.\\
\end{cases}
\end{align}
\end{Prop}

A geometric proof of Proposition \ref{Prop:LDO} can be  found  in \cite{MNS22}. Here,  we provide an algebraic proof. To this end, we need  the \emph{Hecke product} on permutations:
$$
s_i*w=\begin{cases}
    s_iw,& \text{if }s_iw>w,\\[2pt]
    w,& \text{if }s_iw<w,\\
\end{cases}\quad\text{and}\quad
w*s_i=\begin{cases}
    ws_i,& \text{if }ws_i>w,\\[2pt]
    w,& \text{if }ws_i<w.\\
\end{cases}
$$
This  defines a  monoid structure     over $S_\infty$ called the \emph{$0$-Hecke monoid}. 

\begin{proof}[Proof of Proposition \ref{Prop:LDO}]
Without loss of generality,  we may assume $\beta=-1$. 
Suppose that  $w\in S_n$. 
Let $w_0=n\cdots 21$ be the longest element in $S_n$. 
Denote
\begin{align*}
    \G^w(x,t) &= \G_{w_0w}(x,t).
\end{align*}
Then  \eqref{eq:piiG} can be rewritten  as
\begin{align}\label{eq:piGupw}
    \pi_i\G^w = \G^{w*s_i}.
\end{align}
Note that the identity in \eqref{eq:varpiGw} can be restated as 
\begin{align}\label{eq:varpiGupw}
    \varpi_i\G^w = \G^{s_{n-i}*w}.
\end{align}

We  prove \eqref{eq:varpiGupw} by induction on length. 
When $w=\operatorname{id}$, it follows from direct computation that 
\begin{align*}
\varpi_i\G^{\operatorname{id}}(x,t) = 
\varpi_i\G_{w_0}(x,t) 
& = \prod_{\begin{subarray}{c}
a+b\leq n\\
(a,b)\neq (n-i,i)
\end{subarray}} (x_a\ominus t_b),
\end{align*}
which coincides with 
$\G^{s_{n-i}}(x,t)
=\G_{w_0s_{n-i}}(x,t)=\pi_{n-i}\G_{w_0}(x,t)$. 
\bigbreak
For $\ell(w)>0$, 
one can find an index $j$ such that $ws_j<w$, and so by induction, 
\begin{align*}
    \varpi_i\G^{w} & = \varpi_i\pi_{j}\G^{ws_j}
      = \pi_{j}\varpi_i\G^{ws_j}
      = \pi_{j}\G^{s_{n-i}*ws_j}
      = \G^{s_{n-i}*ws_j*s_j}
      = \G^{s_{n-i}*w}.
\end{align*}
Here, we used the fact that the operators  $\pi_j$ and $\varpi_j$ commute in the second equality, and \eqref{eq:piGupw} in the fourth equality.  
\end{proof}


Now, we can give proofs of   Propositions \ref{indonu} and \ref{indonw}. 

\begin{proof}[Proof of Proposition \ref{indonu}]
We introduce another operator 
$$\varphi_i f= -\frac{(1+\beta y_i)f-(1+\beta y_{i+1})f|_{y_i\leftrightarrow y_{i+1}}}{y_i-y_{i+1}},$$
which is the same   as the operator  $\varpi_i$, but acts on the variable  $y$. 
Assume that $s_iu<u$. 
Applying $\varphi_i$ to \eqref{eq:cuvwtoind}, by Proposition \ref{Prop:LDO}, the left-hand side is 
$$
\G_{s_iu}(x,y)\cdot \G_{v}(x,t)=\sum_w c_{s_iu,v}^w(t,y)\cdot \G_{w}(x,t).$$
While the right-hand side is
$$\sum_w \varphi_ic_{u, v}^w(t,y)\cdot \G_{w}(x,t).$$
Comparing the coefficients of $\G_{w}(x,t)$, we are given 
$c_{s_iu,v}^w=\varphi_i c_{u, v}^w$, as desired. 
\end{proof}

\begin{proof}[Proof of Proposition \ref{indonw}]
Apply $\varpi_i$ to \eqref{eq:cuvwtoind}. By Proposition \ref{Prop:LDO}, the left-hand side is 
$$
\begin{cases}
    \G_{u}(x,y)\cdot \G_{s_iv}(x,t)=\sum_w c_{u,s_iv}^w(t,y)\cdot \G_{w}(x,t), & s_iv<v,\\[5pt]
    -\beta\G_{u}(x,y)\cdot \G_{v}(x,t)=\sum_w -\beta c_{u, v}^w(t,y)\cdot \G_{w}(x,t), & s_iv>v.\\
\end{cases}$$
To compute the right-hand side, we use the following property of $\varpi_i$:
\begin{align}\label{eq:LeibnizforLD}
    \varpi_i(fg) & = (f|_{t_i\leftrightarrow t_{i+1}})(\varpi_ig)
-\frac{1+\beta t_i}{t_i-t_{i+1}}(f-f|_{t_i\leftrightarrow t_{i+1}})g.
\end{align}
By \eqref{eq:LeibnizforLD} and Proposition \ref{Prop:LDO},  the right-hand side is 
\begin{align*}
    & \sum_w \varpi_i\big(c_{u, v}^w\cdot \G_{w}(x,t)\big)\\
    & = \sum_{w}\left((c_{u, v}^w|_{t_i\leftrightarrow t_{i+1}})\varpi_i\G_w
    -\frac{1+\beta t_i}{t_i-t_{i+1}}(c_{u, v}^w-c_{u, v}^w|_{t_i\leftrightarrow t_{i+1}})\G_w\right) \\
    & = \sum_{s_iw<w}\left((c_{u, v}^w|_{t_i\leftrightarrow t_{i+1}})\G_{s_iw}
    -\frac{1+\beta t_i}{t_i-t_{i+1}}(c_{u, v}^w-c_{u, v}^w|_{t_i\leftrightarrow t_{i+1}})\G_w\right)\\ 
    &\quad +\sum_{s_iw>w}\left(-\beta(c_{u, v}^w|_{t_i\leftrightarrow t_{i+1}})\G_w
    -\frac{1+\beta t_i}{t_i-t_{i+1}}(c_{u, v}^w-c_{u, v}^w|_{t_i\leftrightarrow t_{i+1}})\G_w\right) \\
    & = -\sum_{s_iw<w} \frac{1+\beta t_i}{t_i-t_{i+1}}(c_{u, v}^w-c_{u, v}^w|_{t_i\leftrightarrow t_{i+1}})\G_w\\
    &\quad + \sum_{s_iw>w}
    \left(c_{u,v}^{s_iw}|_{t_i\leftrightarrow t_{i+1}}
    -\beta c_{u,v}^w|_{t_i\leftrightarrow t_{i+1}}
    -\frac{1+\beta t_i}{t_i-t_{i+1}}(c_{u,v}^w-c_{u,v}^w|_{t_i\leftrightarrow t_{i+1}})\right) \G_w\\
    & = -\sum_{s_iw<w} \frac{1+\beta t_i}{t_i-t_{i+1}}(c_{u,v}^w-c_{u,v}^w|_{t_i\leftrightarrow t_{i+1}})\G_w\\
    &\quad + \sum_{s_iw>w}
    \left(c_{u,v}^{s_iw}|_{t_i\leftrightarrow t_{i+1}}
    -\frac{1+\beta t_i}{t_i-t_{i+1}}c_{u,v}^w
    +\frac{1+\beta t_{i+1}}{t_i-t_{i+1}}c_{u,v}^w|_{t_i\leftrightarrow t_{i+1}}\right) \G_w.
\end{align*}
Extracting  the coefficients of $\G(w)$ with $s_iw>w$ on both sides, we deduce  that 
\begin{equation*}\label{IHVF}
    c_{u,v}^{s_iw}|_{t_i\leftrightarrow t_{i+1}}
=\frac{1+\beta t_i}{t_i-t_{i+1}}c_{u,v}^w
-\frac{1+\beta t_{i+1}}{t_i-t_{i+1}}c_{u,v}^w|_{t_i\leftrightarrow t_{i+1}}
+\begin{cases}
c_{u,s_iv}^w, & s_iv<v,\\[5pt]
-\beta c_{u,v}^w, & s_iv>v,
\end{cases}
\end{equation*}
which coincides with \eqref{PURW-1} after the variable exchange $t_i\leftrightarrow t_{i+1}$.
\end{proof}


\begin{thebibliography}{99}


\bibitem{AGM}
D. Anderson, S. Griffeth and E. Miller, Positivity and Kleiman transversality
in equivariant K-theory of homogeneous spaces, J. Eur. Math. Soc. (JEMS) 13
(2011), 57--84.

\bibitem{BBBG21}
B. Brubaker, V. Buciumas,  D.  Bump and   H. Gustafsson,  Colored five-vertex models and Demazure atoms, J. Combin. Theory Ser. A 178 (2021), Paper No. 105354, 48 pp.



 \bibitem{BR}
 A. Buch and R. Rim\'{a}nyi, Specializations of Grothendieck polynomials, C. R. Acad. Sci.
 Paris, Ser. I 339 (2004), 1--4.

\bibitem{BS22}
V. Buciumas and T. Scrimshaw, 
Double Grothendieck polynomials and colored lattice models,
Int. Math. Res. Not. IMRN,  (2022),
DOI:10.1093/imrn/rnaa327.

\bibitem{BTW20}
V. Buciumas, T. Scrimshaw and K. Weber,  Colored five-vertex models and Lascoux polynomials and atoms,  J. Lond. Math. Soc.   102 (2020),  1047--1066.



\bibitem{Gr01}
W. Graham, Positivity in equivariant Schubert calculus, 
Duke Math. J. 109 (2001),   599–-614.

\bibitem{Huang21}
D. Huang, Schubert products for permutations with separated descents, Int. Math. Res. Not. IMRN (2022), rnac299, https://doi.org/10.1093/imrn/rnac299.

\bibitem{Kirillov07}
A.  Kirillov, Skew divided difference operators and Schubert polynomials, SIGMA Symmetry Integrability Geom. Methods Appl. 3 (2007), Paper 072, 14. 


\bibitem{KT03} A. Knutson and T. Tao, Puzzles and (equivariant) cohomology of Grassmannians, Duke Math. J. 119
(2003),   221--260.




\bibitem{Pzz3}
A. Knutson and P. Zinn-Justin, Schubert puzzles and integrability III: separated descents,  arXiv:2306.13855. 


\bibitem{LLS21}
T. Lam, S.J. Lee and M. Shimozono, Back stable Schubert calculus, Compos. Math.  157 (2021), 883--962.


\bibitem{LLS22}
T. Lam, S.J. Lee and M. Shimozono, Back stable K-theory Schubert calculus, 
Int. Math. Res. Not. IMRN (2022), rnac315, https://doi.org/10.1093/imrn/rnac315.

 
\bibitem{MNS22}
L. Mihalcea, H. Naruse  and C. Su, Left Demazure--Lusztig operators on
equivariant (quantum) cohomology and K-theory, Int. Math. Res. Not. IMRN, 16 (2022):12096–12147.

\bibitem{MS99}
A.I. Molev and B.E. Sagan, A Littlewood--Richardson rule for factorial Schur
functions, Trans. Amer. Math. Soc. 351 (1999), 4429--4443.

\bibitem{PY15}
O. Pechenik and A. Yong, 
Equivariant K-theory of Grassmannians II: The Knutson--Vakil conjecture, Compos. Math. 153 (2017), 667-–677.

\bibitem{WZ19}
M. Wheeler and P. Zinn-Justin, Littlewood--Richardson coefficients for Grothendieck polynomials from integrability, J. Reine Angew. Math.  757 (2019), 159--195.


\bibitem{Vakil06}
R. Vakil, A geometric Littlewood--Richardson rule, Ann. of Math.  164 (2006), 371--421.

\bibitem{Weigandt21}
A. Weigandt, Bumpless pipe dreams and alternating sign matrices, J. Combin. Theory 
Ser. A 182 (2021), 105470.

\bibitem{ZJ2009}
P. Zinn-Justin, Littlewood--Richardson coefficients and integrable tilings, Electron. J. Combin. 16 (2009),
Research Paper 12.

\bibitem{ZJ09}
P. Zinn-Justin, Six-vertex, Loop and Tiling Models: Integrability and Combinatorics, arXiv:0901.0665v2.

\end{thebibliography}
\end{document}